\documentclass[11pt,reqno]{amsart}
\usepackage[margin=1in]{geometry}
\usepackage{amsmath, amsthm, amssymb,enumitem}
\usepackage{hyperref}
\hypersetup{colorlinks=true, pdfstartview=FitV, linkcolor=blue,citecolor=blue, urlcolor=blue}
\usepackage[abbrev,lite,nobysame]{amsrefs}
\usepackage{times}
\usepackage[usenames,dvipsnames]{color}
\usepackage{mathtools}

\definecolor{labelkey}{rgb}{0,0,1}
\usepackage{dsfont}
\allowdisplaybreaks
\newtheorem{proposition}{Proposition}[section]
\newtheorem{theorem}[proposition]{Theorem}
\newtheorem{corollary}[proposition]{Corollary}
\newtheorem{lemma}[proposition]{Lemma}
\newtheorem*{bootstrap*}{Bootstrap Step}
\theoremstyle{definition}

\newtheorem{remark}[proposition]{Remark}

\newtheorem{assumption}[proposition]{Assumption}
\numberwithin{equation}{section}

\def\Re{{\rm Re}}
\def\Im{{\rm Im}}

\title[Enhanced dissipation and blow-up suppression]{Enhanced dissipation and blow-up suppression for an aggregation\\ equation with fractional diffusion and shear flow}

\author[Binqian Niu]{Binqian Niu}
\address{\scriptsize School of Mathematical Sciences, Shanghai Jiao Tong University, Shanghai, 200240, P.R.China.}
\email{bqbling@sjtu.edu.cn}
\author[Binbin Shi]{Binbin Shi}
\address{\scriptsize School of Mathematics and Statistics, Nanjing University of Science and Technology, Nanjing, 210094, P.R.  China.}
\email{shibb@njust.edu.cn}
\author[Weike Wang]{Weike Wang}
\address{\scriptsize School of Mathematical Sciences, CMA-Shanghai and Institute of Natural Science, Shanghai Jiao Tong  University, Shanghai, 200240, P.R.China.}
\email{wkwang@sjtu.edu.cn}


\begin{document}

\begin{abstract}
In this paper, we consider an aggregation equation with fractional diffusion and large shear flow, which arise from modelling chemotaxis in bacteria. Without the advection, the solution of aggregation equation may blow up in finite time. First, we study the enhanced dissipation of shear flow by resolvent estimate method, where the fractional Laplacian $(-\Delta)^{\alpha/2}$ is considered and $\alpha\in (0,2)$. Next, we show that the enhanced dissipation of shear flow can suppress blow-up of solution to aggregation equation with fractional diffusion and establish global classical solution in the case of $\alpha\geq 3/2$. Here we develop  some new technical to overcome the difficult of low regularity for fractional Laplacian.
\end{abstract}


\maketitle

\vspace{-1em}

\section{Introduction}\label{sec.1}

Aggregation-diffusion-type equations arise in a wide variety of biological applications, such as Keller-Segel models of chemotaxis and migration patterns in ecological systems. In this paper, we consider the following aggregation equation on torus $\mathbb{T}^2$ with fractional diffusion and large shear flow
\begin{equation}\label{eq:1.1}
\begin{cases}
\partial_tn+Au(y)\partial_xn+(-\Delta)^{\alpha/2} n+\nabla\cdot\left(n \mathbf{B}(n)\right)=0,\\
n(t,x,y)\big|_{t=0}=n_0(x,y),\ \ \ (t,x,y)\in \mathbb{R}^{+}\times\mathbb{T}^2.
\end{cases}
\end{equation}
Here the $n(t,x,y)$ is nonnegative unknown functions which represent the density, the smooth function $u(y)$ represents the underlying fluid velocity and $A$ is a positive constant. The domain
$$
\mathbb{T}^2=\{(x,y)\big|x,y\in \mathbb{T}_{2\pi}\},
$$
where $\mathbb{T}_{2\pi}=[-\pi,\pi)$ is a periodic interval. The fractional Laplacian  $(-\Delta)^{\alpha/2}$ is defined via the Fourier transform, it is as follows
\begin{equation}\label{eq:1.2}
(-\Delta)^{\alpha/2} n=\sum_{(k,l)\in \mathbb{Z}^{2}}(k^2+l^2)^{\alpha/2}\widehat{n}(k,l)e^{ikx+ily},\ \ \ 0<\alpha\leq2,
\end{equation}
the $\widehat{n}$ denotes the Fourier transform of $n$. The fractional Laplacian is a nonlocal operator and its kernel representation is used in this paper, see Section \ref{sec.2}. The linear vector operator $\mathbf{B}(n)$ is called attractive kernel, which could be formally represented as
\begin{equation}\label{eq:1.3}
\mathbf{B}(n)=\nabla (-\Delta)^{-1}(n-\overline{n}),
\end{equation}
the $\overline{n}$ denotes the average of $n$. In this paper, we study the global well-posedness of equation \eqref{eq:1.1} for some large shear flows.

\vskip .05in

Without the advection ($A=0$), the equation \eqref{eq:1.1} is an aggregation equation with fractional diffusion
\begin{equation}\label{eq:1.4}
\partial_tn+(-\Delta)^{\alpha/2} n+\nabla\cdot\left(n \mathbf{B}(n)\right)=0,
\end{equation}
which is used as a model for various biological and physical phenomena, see \cite{BW.1999,BPK.1999}. When $\alpha=2$, the equation \eqref{eq:1.4} goes back to the classical parabolic-elliptic Keller-Segel model. It is well-known that the solutions for Keller-Segel model in high dimensional may blow up in finite time if the initial data $n_0$ is large in $L^1$ norm.  More precisely, for two dimensional case, if the $L^1$ norm of initial data {{$n_0$}} is less than $8\pi$, there exists a unique global solution; and if the $L^1$ norm of initial data {{$n_0$}} exceeds $8\pi$, the solution may blow up in finite time, one could refer to \cite{HV.1996,HV.1997,Nagai.1995} for more details. In three and higher dimensional cases, the solution may blow up even for initial data $n_0$ arbitrary small in $L^1$ norm, see \cite{CPZ.2004,SW.2019}. When $0<\alpha<2$, the solution of \eqref{eq:1.4} may blow up in finite time for high dimensions and large initial data $n_0$, see \cite{BK.2010,LRZ.2010}. In addition, the blow-up solution has also been studied in \cite{BK.2010,LR.2009,LS.2019} for more general aggregation-diffusion equations.

\vskip .05in

The case $A\neq0$ is corresponding to aggregation progress in the background of a shear flow. A realistic scenario is that chemotactic processes take place in a moving fluid, and the possible effects and related problems resulting from the interactions between the chemotactic process and the fluid transport have been widely investigated, see \cite{DLM.2010,LL.2011,Lorz.2010,Winkler.2012,CCDL.2016,WW.2019}. The study of aggregation equation with an incompressible flow is one of those attempts, the model is as follows
\begin{equation}\label{eq:1.5}
\partial_tn+A{\bf u}\cdot\nabla n+(-\Delta)^{\alpha/2} n+\nabla\cdot\left(n \mathbf{B}(n)\right)=0,
\end{equation}
where ${\bf u}$ is divergence free vector field. An interesting question arising is whether one can suppress the possible finite time blow-up by the mixing effect coming from the fluid transport. Recently, some progresses have been made for the suppression of blow-up by incompressible flow. When $\alpha=2$, the equation \eqref{eq:1.5} is classical Keller-Segel model with incompressible flow. Kiselev, Xu \cite{KX.2016} and Hopf, Rodrigo \cite{HR.2018} considered that the ${\bf u}$ is the relaxation enhancing flow \cite{CKRZ.2008}, they proved that the solution of the advective Keller-Segel equation does not blow-up in finite time provided the amplitude of the relaxation enhancing flow is large enough. Bedrossian and He \cite{BH.2017} found that shear flows (${\bf u}=(u(y),0)$) have a different suppression effect in the sense that sufficiently large shear flows could prevent the blow-up in two dimensions but could not guarantee the global existence in three dimensions if the initial mass is greater than $8\pi$. When  $0<\alpha<2$, the equation \eqref{eq:1.5} is an aggregation equation with fractional diffusion (also known as generalized Keller-Segel model) with incompressible flow. Hopf, Rodrigo \cite{HR.2018} and Shi, Wang \cite{SW.2020} proved that the solution of \eqref{eq:1.5} does not blow-up in finite time by the large relaxation enhancing flow, where the range of $\alpha$ need to be considered. For the blow-up phenomenon can be suppressed through fluid transport progress, some other problems and models can refer to \cite{He.2018,FSW.2022,ZZZ.2021,He.2023,DT.2024,LXX.2023,CW.2023,SW.2023}. However, the equation \eqref{eq:1.5} becomes \eqref{eq:1.1} if ${\bf u}$ is shear flow, and it is currently unclear whether the shear flow can suppress the blow-up in the case of $0<\alpha<2$.

\vskip .05in

The additional flows studied in those references are found to provide an enhanced dissipation effect from fluid, which could help the dissipation terms dominate even in the nonlinear level. In this paper, we study that the blow-up solution of \eqref{eq:1.1} can  be suppressed by enhanced dissipation of shear flow. First, we need to consider the enhanced dissipation of shear flow in the case of fractional dissipation, the model is as follows
\begin{equation}\label{eq:1.6}
\partial_tg+u(y)\partial_x g+\nu(-\Delta)^{\alpha/2}g=0,\ \ \ g(0,x,y)=g_0(x,y),
\end{equation}
where $\nu>0$. The meaning of enhanced dissipation is that the dissipation effect can be enhanced and the $L^2$ norm of solution to equation \eqref{eq:1.6} has a faster decaying rate if $\nu$ is small enough. When $\alpha=2$, Bedrossian and Zelati \cite{BC.2017} studied the enhanced dissipation of \eqref{eq:1.6} by hypocercivity, and Wei \cite{Wei.2021} studied the enhanced dissipation of \eqref{eq:1.6} by resolvent estimate. Recently, some special shear flows have been widely studied, such as Couette flow \cite{BMV.2016}, Poiseuille flow \cite{CEW.2020} and Kolmogorov flow \cite{Wei.201901,Wei.2020}. When $0<\alpha<2$,~Zelati, Delgadino and Elgindi \cite{CDE.2020} given a enhanced dissipation rate of \eqref{eq:1.6}. He \cite{He.2022} obtained an almost sharp enhanced dissipation rate of \eqref{eq:1.6} by resolvent estimate, where $\alpha>1$ and the fractional Laplacian operator was written as the form of anisotropic. Li and Zhao \cite{LZ.2023} considered the linearized critical surface quasi-geostrophic equation around the Kolmogorov flow. As a toy model, they studied the equation \eqref{eq:1.6} in the case of $\alpha=1$ and $u(y)=\cos y$, and obtained the sharp enhanced dissipation rate by hypocercivity. In \cite{He.2022,LZ.2023}, the authors also given some useful comments for the range of $\alpha$.

\vskip .05in

In this paper, we consider the equation \eqref{eq:1.1}, the goal is to show that the blow-up solution of an aggregation  equation with fractional diffusion can be suppressed through some shear flows. We will prove the global existence of solution to aggregation equation with fractional diffusion and large shear flow for some $\alpha$. This question is motivated by the works of \cite{BH.2017,HR.2018,SW.2020}. Here we study the sharp enhanced dissipation rate of \eqref{eq:1.6} in the case of $\alpha>0$. As an application, we show that the enhanced dissipation of shear flow can suppressed the blow-up of an aggregation  equation with fractional diffusion. 

\vskip .05in

First, we study the enhanced dissipation of shear flow in an aggregation  equation with fractional diffusion, the equation \eqref{eq:1.6} is written as
\begin{equation}\label{eq:1.7}
\partial_tg+\mathcal{L}_{\nu,\alpha} g=0,\ \ g(0,x,y)=g_0(x,y),
\end{equation}
where
\begin{equation}\label{eq:1.8}
\mathcal{L}_{\nu,\alpha}=\nu(-\Delta)^{\alpha/2}+u(y)\partial_x, \ \ 0<\alpha<2.
\end{equation}
In this paper, we study the enhanced dissipation of shear flow by the equation \eqref{eq:1.7}. In fact, we only to study semigroup with the operator $-\mathcal{L}_{\nu,\alpha}$ as generator and give the semigroup estimate. Let denote $P_0$ and $P_{\neq}$ are  projection operator, which are defined that for any $g(x,y)$
\begin{equation}\label{eq:1.9}
g^0=P_0g=\frac{1}{2\pi}\int_{\mathbb{T}}g(x,y)dx,\ \ \ \ g_{\neq}=P_{\neq}g=g-g^0,
\end{equation}
where $g^0$ is zero node and $g_{\neq}$ is nonzero mode.

\vskip .05in

The $e^{-t\mathcal{L}_{\nu,\alpha}}$ is semigroup with the operator $-\mathcal{L}_{\nu,\alpha}$ as generator, the first main theorem  is the following

\begin{theorem}\label{thm:1.1}
Assume that the shear flow $u(y)\in C^{\infty}(\mathbb{T})$, there exist constants $m, N\in \mathbb{N}$, $c_1>0$, $\delta_0>0$ with the property that, for
any $\lambda \in \mathbb{R}$ and $\delta\in(0,\delta_0)$, there exist finitely many points $y_1,\ldots y_n\in \mathbb{T}$ with $n\leq N$, such that
$$
|u(y)-\lambda|\geq c_1 \delta^m, \qquad \forall \  |y-y_j|\geq \delta, \quad \forall j\in \{1,\ldots n\}.
$$
Then there exists a $\nu_0=\nu(u)$ such that if $\nu<\nu_0$, for any $t\geq0$, the following enhanced dissipation estimate hold,
$$
\left\|e^{-t\mathcal{L}_{\nu,\alpha}}P_{\neq}\right\|_{L^2\rightarrow L^2}\leq e^{-\lambda'_{\nu,\alpha} t+\pi/2},\ \ \ \ \ \lambda'_{\nu,\alpha}=\epsilon_0 \nu^{\frac{m}{m+\alpha}},
$$
where the $\mathcal{L}_{\nu,\alpha}$ and $P_{\neq}$ are defined in \eqref{eq:1.8} and \eqref{eq:1.9} respectively, the $\epsilon_0$ is small enough and $\alpha\in (0,2)$.
\end{theorem}

In Theorem \ref{thm:1.1}, we obtain the sharp enhanced dissipation rate of shear flow, this result seems to improve previous works \cite{CDE.2020,He.2022,LZ.2023}, where general shear flow $u(y)$ satisfies the assumption in Theorem \ref{thm:1.1} and $\alpha\in (0,2)$.  In this paper, we study enhanced dissipation of shear flow based on the assumption in Theorem \ref{thm:1.1} and resolvent estimate, which is inspired by the works of \cite{CDFM.2021,Wei.2021,FSW.2022}.

\vskip .05in

The assumption of shear flow in Theorem \ref{thm:1.1}
is reasonable, we can easily check that the Kolmogorov flow $u(y)=\cos y$ satisfies the assumption by Taylor expansion and $m=2$. In the case of $\alpha=2, u=\cos y$, the enhanced dissipate rate has been studied in \cite{Wei.201901,Wei.2020}. If we consider the case of $0<\alpha<2$ and $u(y)=\cos y$, we also have following corollary by Theorem \ref{thm:1.1},

\begin{corollary}\label{cor:1.2}
Assume that the shear flow $u(y)=\cos y$ is Kolmogorov flow, then there exists a $\nu_0=\nu(u)$ such that if $\nu<\nu_0$, for any $t\geq0$, the following enhanced dissipation estimate hold,
\begin{equation}\label{eq:1.10}
\left\|e^{-tL_{\nu,\alpha}}P_{\neq}\right\|_{L^2\rightarrow L^2}\leq e^{-\lambda_{\nu,\alpha} t+\pi/2},\ \ \ \ \ \lambda_{\nu,\alpha}=\epsilon_0 \nu^{\frac{2}{2+\alpha}},
\end{equation}
where
\begin{equation}\label{eq:1.1101}
L_{\nu,\alpha}=\nu(-\Delta)^{\alpha/2}+\cos y\partial_x,
\end{equation}
the $P_{\neq}$ are defined in \eqref{eq:1.9}, the $\epsilon_0$ is small enough and $\alpha\in (0,2)$.
\end{corollary}

\vskip .05in

Next, we study that the enhanced dissipation of shear flow can suppress blow up in an aggregation equation with fractional diffusion, and we establish the global classical solution with large initial data. Here we study the enhanced dissipation of \eqref{eq:1.1} by operator $\mathcal{L}_{\nu,\alpha}$ and Theorem \ref{thm:1.1}, thus we modify the  \eqref{eq:1.1} by time rescaling. If taking $t=A\tau$ and denoting $\nu=A^{-1}$, the equation \eqref{eq:1.1} is written as
\begin{equation}\label{eq:1.11}
\begin{cases}
\partial_tn+u(y)\partial_xn+\nu(-\Delta)^{\alpha/2} n+\nu\nabla\cdot\left(n \mathbf{B}(n)\right)=0,\\
n(t,x,y)\big|_{t=0}=n_0(x,y),\ \ \ (t,x,y)\in \mathbb{R}^{+}\times\mathbb{T}^2.
\end{cases}
\end{equation}
Since equations \eqref{eq:1.1} and \eqref{eq:1.11} are equivalent in the sense of time rescaling, we mainly consider the \eqref{eq:1.11} and establish the global classical solution by enhanced dissipation of shear flow. Based on the Theorem \ref{thm:1.1}, we imply that
$$
\nu^{\frac{m}{m+\alpha}}\rightarrow \nu, \ \ \ m\rightarrow\infty,
$$
then the enhanced dissipation of shear  flow is very weak for $m$ is large enough, it is difficult for suppressing the blow-up. In this paper, we consider that the shear flow $u(y)=\cos y$ is Kolmogorov flow, and we imply $m=2$ to the assumption in Theorem \ref{thm:1.1}.

\vskip .05in

The second main theorem of this paper read as follow.

\begin{theorem}\label{thm:1.2}
Let $\alpha\in [3/2,2)$ and initial data $n_0\geq0, n_0\in H^\gamma(\mathbb{T}^2)\cap L^1(\mathbb{T}^2), \gamma>1+\alpha$.  the $u(y)=\cos y$ is Kolmogorov flow. Then there exists a $\nu_0=\nu(\alpha,n_0)$, such that if $\nu<\nu_0$, the unique nonnegative classical solution $n(t,x,y)$ of equation \eqref{eq:1.11} is global in time.
\end{theorem}

In Theorem \ref{thm:1.2}, we show that the shear flow can suppress the blow-up of aggregation equation with fractional diffusion in the case of $\alpha\geq 3/2$, it seems to be the first result about shear flow suppress blow up in nonlinear equation with fractional diffusion. I believe that the range of $\alpha$ is not sharp, which is only technical one. We know by \eqref{eq:1.4} that the aggregation equation with fractional diffusion  is scaling invariant, thus the critical space of \eqref{eq:1.4} is $L^{d/\alpha}$, where $d$ represents the spatial dimension. It is well know that the solution is global smooth if we have global supercritical estimate. By the definition of $P_0, P_{\neq}$ and Theorem \ref{thm:1.1}, we deduce that the zero mode of \eqref{eq:1.11} is one dimension which is no enhanced dissipation, and the $L^1$ is supercritical for $\alpha>1$. Therefore, I speculate that the Theorem \ref{thm:1.2} still holds in the case of $\alpha>1$. In addition, the $L^2$ estimate of \eqref{eq:1.11} is supercritical in the case of $\alpha\geq 3/2$, thus the key point of this paper is to establish global $L^2$ estimate of \eqref{eq:1.11} through the enhanced dissipation of shear flow.

\vskip .05in

If consider the equation \eqref{eq:1.1}, we have the following corollary by Theorem \ref{thm:1.2}.

\begin{corollary}\label{cor:1.3}
Let $\alpha\in [3/2,2)$ and initial data $n_0\geq0, n_0\in H^\gamma(\mathbb{T}^2)\cap L^1(\mathbb{T}^2), \gamma>1+\alpha$, the $u(y)=\cos y$ is Kolmogorov flow. Then there exists a $A_0=A(\alpha,n_0)$, such that if $A>A_0$, the unique nonnegative classical solution $n(t,x,y)$ of equation \eqref{eq:1.1} is global in time.
\end{corollary}

In the following, we briefly state our main ideas of the proofs to Theorem \ref{thm:1.1} and \ref{thm:1.2}. Firstly, we need to proved the enhanced dissipation of equation \eqref{eq:1.6} by resolvent estimate and semigroup estimate, see Theorem \ref{thm:1.1}. We assume that the shear flow satisfies the assumption in Theorem \ref{thm:1.1} and consider the operator $\mathcal{L}_{\nu,\alpha}$ in \eqref{eq:1.8}. In the proof, we establish the pseudospectral bound estimate by resolvent estimate and can obtain the semigroup estimate of $e^{-t\mathcal{L}_{\nu,\alpha}}$ by the Gearchart-Pr\"{u}ss type theorem (see Lemma \ref{lem:2.1}). Next, we prove that the shear flow suppress the blow-up of equation \eqref{eq:1.11} and establish global classical solution, see Theorem \ref{thm:1.2}. Since $L^2$ is supercritical estimate in the case of $\alpha>3/2$, we only need to establish global $L^2$ estimate. We decompose equation \eqref{eq:1.11} into one dimensional zero mode equation and two dimensional nonzero mode equation, see \eqref{eq:2.3}-\eqref{eq:2.5}. Since $L^1$ is supercritical in one dimension zero mode equation, the solution is global existence. And two dimensional nonzero mode equation has enhanced dissipation, then we can establish global $L^2$ estimate by bootstrap argument.

\vskip .05in

In this paper, we study the enhanced dissipation of shear flow in \eqref{eq:1.6} and blow-up suppression in an aggregation equations with fractional diffusion, see equation \eqref{eq:1.11}. Some proofs technical and ideas are inspired by \cite{BH.2017,FSW.2022}. Here we use semigroup theory to study the enhanced dissipation of \eqref{eq:1.6} and the key point is to obtain  pseudospectral bound by resolvent estimate method. Since the fractional Laplacian is nonlocal operator in the case of $\alpha\in(0,2)$, these technical is not obvious. Our strategy is to transform the operator $\mathcal{L}_{\nu,\alpha}$ into the case similar to $\alpha=2$ through transformation and calculation, the details can refer to Section \ref{sec.3}. We study that the shear flow suppress the blow-up of \eqref{eq:1.11} by bootstrap argument, where we only consider the nonzero mode equation, see Assumption \ref{ass:2.8} and Proposition \ref{prop:2.9}. Here the estimates of \eqref{eq:1.11} and \eqref{eq:2.3} are also needed, see Lemma \ref{lem:4.1},  \ref{lem:4.3} and  \ref{lem:4.4}. Some mathematical methods are used in the proof, such as energy methods, nonlinear maximum principle, semigroup theory and Duhamel's principle, the details can refer to Section \ref{sec.4}. Compared to the case of $\alpha=2$, the enhanced dissipation estimate of $n_{\neq}$ is difficult in the case of $\alpha<2$, the main reason is that the low regularity of dissipative terms cannot control nonlinear terms in the \eqref{eq:1.11}. In this paper, we establish some new techniques to overcome this difficulty, the details can refer to the proof of Proposition \ref{prop:2.9} and Appendix \ref{sec.B}.

\vskip .05in

The rest of this paper is arranged as follows. In Section \ref{sec.2}, we introduce some preparations and give the bootstrap argument. In Section \ref{sec.3}, we prove the enhanced dissipation effect of shear flow. In Section \ref{sec.4}, we prove  the Theorem \ref{thm:1.2} to establish the global well-posedness of. In the Appendix \ref{sec.A} and \ref{sec.B}, we provide necessary supplements and useful tools in this paper.

\section{Preliminaries and  Bootstrap argument}\label{sec.2}

In what follows, we provide some  notations and auxiliary results, which is helpful for the proof of this paper. In addition, we set up the bootstrap argument in this section. The details are as follows.

\subsection{Notations and auxiliary results}
Throughout the paper, we use the standard notations to denote function spaces and use $C$ to denote a generic constant which may vary from line to line. Given quantities $X,Y$, if there exists a positive constant $C$ such that $X\leq CY$, we write $X\lesssim Y$. If there exist positive constants $C_1,C_2$ such that $C_1Y\leq X\leq C_2Y$, we write $X\sim Y$.

\vskip .05in

In this paper, we study the enhanced dissipation of shear flow by resolvent estimate, where the fractional Laplacian is considered. Therefore, we need to introduce some the operator theory. Let $(\mathcal{X},\|\cdot\|)$ be a complex Hilbert space and let $H$ be a closed linear operator in $\mathcal{X}$ with domain $D(H)$. $H$ is m-accretive if the left open half-plane is contained in the resolvent set with
$$
(H+\lambda I)^{-1}\in \mathcal{B}(\mathcal{X})\,,\quad \left\|(H+\lambda I)^{-1}\right\|\leq (\Re \lambda)^{-1},\ \  \Re\lambda >0,
$$
where $\mathcal{B}(\mathcal{X})$ denotes the set of bounded linear operators on $\mathcal{X}$ with operator norm $\|\cdot\|$ and $I$ is the identity operator.

\vskip .05in

We denote $e^{-tH}$ is a  semigroup with $-H$ as generator and define pseudospectral bound
\begin{equation}\label{eq:2.1}
\Psi(H)=\inf\{\|(H-i\lambda I)f\|: f\in D(H), \lambda\in \mathbb{R}, \|f\|=1 \}.
\end{equation}
The following result is the Gearchart-Pr\"{u}ss type theorem for m-accretive operators, see \cite{Wei.2021}.

\begin{lemma}\label{lem:2.1}
Let $H$ be an m-accretive operator in a Hilbert space $\mathcal{X}$. Then for any $t\geq 0$, we have
$$
\left\|e^{-tH}\right\|_{\mathcal{X}\rightarrow \mathcal{X}}\leq e^{-t\Psi(H)+\pi/2},
$$
where $\Psi(H)$ is defined in \eqref{eq:2.1}.
\end{lemma}

The fractional Laplacian in \eqref{eq:1.2} is a nonlocal operator, it also has the following kernel representation on $\mathbb{T}^d$, see \cite{CC.2004}
\begin{equation}\label{eq:2.2}
\Lambda^\alpha f(x)=C_{\alpha,d}\sum_{k\in\mathbb{Z}^d}P.V. \int_{\mathbb{T}^d}\frac{f(x)-f(y)}{|x-y+k|^{d+\alpha}}dy,\
\end{equation}
where $\Lambda=(-\Delta)^{1/2}, \alpha\in (0,2), x,y\in \mathbb{T}^d, C_{\alpha,d}>0$. In this paper, we denote $\Lambda_x,\Lambda_y$ as one dimension fractional Laplacian operator. Next, we present some lemmas related to the fractional Laplacian,  which is helpful in this paper, the details are as follows. 
\begin{lemma}[\cite{DDL.2009}] \label{lem:2.2}
Let $\alpha\in[0,2]$, for any $f, g\in C^\infty(\mathbb{T}^d)$, one has
$$
\int_{\mathbb{T}^d}\Lambda^\alpha f(x)g(x)dx=\int_{\mathbb{T}^d} f(x)\Lambda^\alpha g(x)dx.
$$
\end{lemma}

\begin{lemma}[Nonlinear maximum principle \cite{Granero.2016,SW.2020,CV.2012}]\label{lem:2.3}
Let $\alpha\in(0,2), f\in C^\infty(\mathbb{T}^d)$ and denote by $\overline{x}$ the point such that
$$
f(\overline{x})=\max_{x\in \mathbb{T}^d}f(x),
$$
and $f(\overline{x})>0$. Then for any $1\leq p<\infty$, we have
$$
\Lambda^{\alpha}f(\overline{x})\geq C(\alpha,d,p)\frac{f(\overline{x})^{1+p\alpha/d}}{\|f\|_{L^p}^{p\alpha/d}},\ \ \
or\ \ \
f(\overline{x})\leq C(d,p)\big\|f\big\|_{L^p}.
$$
\end{lemma}

\begin{lemma}[Kato-Ponce inequality \cite{GO.2014}]\label{lem:2.4}
Let $s\geq0, p\in(1,\infty)$, for any $f, g\in C^\infty(\Omega)$, one has
$$
\big\|\Lambda^{s}(fg)\big\|_{L^p}\lesssim \big\|\Lambda^sf\big\|_{L^{p_1}}\big\|g\big\|_{L^{p_2}}+\big\|\Lambda^sg\big\|_{L^{p_3}}\big\|f\big\|_{L^{p_4}},
$$
where $1<p_{i}\leq\infty\ (i=1,2,3,4)$, $1/p=1/p_1+1/p_2=1/p_3+1/p_4$.
\end{lemma}

\begin{lemma}[\cite{SYK.2015}]\label{lem:2.01}
Let $\alpha\in(0,1)$,  for any $f\in C^\infty(\mathbb{T}^d)$ and $\varepsilon>0$, one has
$$
\left\|\Lambda^\alpha f\right\|_{L^\infty}\lesssim \varepsilon^{1-\alpha}\|\nabla f\|_{L^\infty}+\varepsilon^{-\alpha}\|f\|_{L^\infty}.
$$
\end{lemma}

\begin{lemma}\label{lem:2.5}
Let $f,g\in C^\infty(\mathbb{T}^d)$ satisfy $\widehat{f}(0)\widehat{g}(0)=0$, then for any $s\in [0,2]$, one has
\begin{equation}\label{eq:2.304}
\int_{\mathbb{T}^d} f(x)g(x)dx=\int_{\mathbb{T}^d} \Lambda^{-s}f(x)\Lambda^{s} g(x)dx,
\end{equation}
and
\begin{equation}\label{eq:2.306}
\left|\int_{\mathbb{T}^d}f(x)g(x)dx\right|\lesssim \big\|\Lambda^{s}f\big\|_{L^2}\left\|\Lambda^{-s}g\right\|_{L^2},
\end{equation}
where $\widehat{f}$ and $\widehat{g}$ are the Fourier transform.
\end{lemma}

\begin{remark}\label{rem:2.9}
In Lemma \ref{lem:2.5}, combining $\widehat{f}(0)\widehat{g}(0)=0$, one has 
$$
\int_{\mathbb{T}^d}f(x)g(x)dx=\sum_{k\in \mathbb{Z}^d}\widehat{f}(k)\widehat{g}(k)=\sum_{k\neq 0}|k|^{-s}\widehat{f}(k)|k|^{s}\widehat{g}(k)
=\int_{\mathbb{T}^d} \Lambda^{-s}f(x)\Lambda^{s} g(x)dx.
$$
and the \eqref{eq:2.306} was established in \cite{HR.2018}.
\end{remark}

In this paper, we only need to establish global $L^2$ estimate of \eqref{eq:1.11}, the main reason is that there are the following local existence and regularity criterion.

\begin{proposition}\label{prop:2.6}
Let $\alpha\in[3/2,2)$ and initial data $n_0\geq0, n_0\in H^{\gamma}(\mathbb{T}^2)\cap L^{1}(\mathbb{T}^2), \gamma>1+\alpha$, the shear flow $u(y)=\cos y$ is Kolmogorov flow, there exists a time $T_*=T(n_0,\alpha,\nu)>0$ such that the nonnegative solution of \eqref{eq:1.11}
$$
n(t,x,y)\in C([0,T_*], H^\gamma(\mathbb{T}^2)\cap L^{1}(\mathbb{T}^2)).
$$
Moreover, if for a given $T$, the solution of \eqref{eq:1.11} verifies the following bound
$$
\lim_{t\to T}\sup_{0\leq\tau\leq t}\big\|n(\tau,\cdot)\big\|_{L^2}<\infty,
$$
then the solution can be extended up to time $T+\delta$ for sufficiently small $\delta>0$. If $n_0\in L^1(\mathbb{T}^2)$, the solution of \eqref{eq:1.11} is $L^1$ conservation, namely,
$$
M=\big\|n\big\|_{L^1}=\big\|n_0\big\|_{L^1}.
$$
\end{proposition}

\begin{remark}\label{rem:2.7}
The Proposition \ref{prop:2.6} tell us that we only need to have certain control of spatial $L^2$ norm of the solution for establishing the classical solution of equation \eqref{eq:1.11}. The proofs of local existence and $L^1$ conservation is standard method, the proof of regularity criterion can refer to \cite{HR.2018,KX.2016}.
\end{remark}

\subsection{Bootstrap argument} We know that the enhanced dissipation of shear flow occurs in nonzero mode. Similar to \cite{BH.2017,FSW.2022}, denote
$$
n^{0}=P_0n,\ \  n_{\neq}=P_{\neq}n,
$$
where the $P_0$ and $P_{\neq}$ are defined in \eqref{eq:1.9}, the solution of equation \eqref{eq:1.11} is decomposed  into $x$-independent part and $x$-dependent part. Since
$$
(-\Delta)^{\alpha/2} n=\sum_{(k,l)\in \mathbb{Z}^{2}}(k^2+l^2)^{\alpha/2}\widehat{n}(k,l)e^{ikx+ily},
$$
and combining the definition of  $P_0$ and $P_{\neq}$ in \eqref{eq:1.9}, one gets 
$$
P_0((-\Delta)^{\alpha/2} n)=(-\partial_{yy})^{\alpha/2}n^0,\ \ \
P_{\neq}((-\Delta)^{\alpha/2} n)=(-\Delta)^{\alpha/2} n_{\neq}.
$$
Then we obtain the one dimensional zero mode equation
\begin{equation}\label{eq:2.3}
\partial_tn^{0}+\nu(-\partial_{yy})^{\alpha/2}n^0+\nu\partial_y(n^{0}\mathbf{B}_1(n^0))+\nu\left(\nabla\cdot(n_{\neq} \mathbf{B}_2(n_{\neq}))\right)^{0}=0,
\end{equation}
and the two dimensional nonzero mode equation
\begin{equation}\label{eq:2.4}
\begin{aligned}
\partial_tn_{\neq}+u(y)\partial_{x}n_{\neq}&+\nu(-\Delta)^{\alpha/2} n_{\neq}+\nu\nabla n^{0}\cdot \mathbf{B}_2(n_{\neq})+\nu\partial_yn_{\neq}\mathbf{B}_1(n^0)\\
& +\nu\left(\nabla\cdot(n_{\neq} \mathbf{B}_2(n_{\neq}))\right)_{\neq}-\nu n^{0}n_{\neq}-\nu n_{\neq}(n^0-\overline{n})=0,
\end{aligned}
\end{equation}
where
\begin{equation}\label{eq:2.5}
\mathbf{B}_1(n^0)=\partial_y(-\partial_{yy})^{-1}(n^0-\overline{n}), \ \ \ \mathbf{B}_2(n_{\neq})=\nabla(-\Delta)^{-1}n_{\neq}.
\end{equation}
By local estimate (see Lemma \ref{lem:A.301}), we know that there exists
$$
t_0=t(n_0, \alpha, \nu)=O(1/\nu),
$$
such that for any $0\leq t\leq t_0$, one has
$$
\big\|n\big\|^2_{L^2}=\big\|n^0\big\|^2_{L^2}+\big\|n_{\neq}\big\|^2_{L^2}\leq 4\big\|n_0\big\|^2_{L^2},
$$
where the definition of $t_0$ can be seen in \eqref{eq:A.501}.

\vskip .05in

In this paper, we establish the global well-posedness based on the standard bootstrap argument. We denote
\begin{equation}\label{eq:2.801}
s_0=t_0/2,
\end{equation}
then we list the bootstrap assumptions as below.

\begin{assumption}\label{ass:2.8}
Let $\alpha\in [3/2,2)$ and $n_{\neq}$ be the solution of \eqref{eq:2.4} with initial data $n_0$, the positive constant $C_{\neq}$ be determined by the proof.~Assume that the $u(y)=\cos y$ is Kolmogorov flow, define $T^\ast=T(n_0,\alpha,\nu)>0$ to be the maximum time such that following assumptions hold,
\begin{itemize}
\item [(A-1)] Nonzero mode $L^2\dot{H}^{\alpha/2}$ estimate: for any $s_0\leq s\leq t\leq T^\ast$
$$
\nu\int_{s}^{t}\left\|\Lambda^{\alpha/2} n_{\neq}\right\|^2_{L^2}d\tau\leq 16C_{\neq}e^{-\lambda_{\nu,\alpha} (s-s_0)}\big\|n_{0}\big\|^2_{L^2},
$$
\item [(A-2)]Nonzero mode enhanced dissipation estimate: for any $s_0\leq t\leq T^\ast$
$$
\big\|n_{\neq}(t)\big\|^2_{L^2}\leq 4C_{\neq}e^{-\lambda_{\nu,\alpha} (t-s_0)}\big\|n_{0}\big\|^2_{L^2},
$$
\end{itemize}
where $\lambda_{\nu,\alpha}$ is defined in \eqref{eq:1.10} and $s_0$ is defined in \eqref{eq:2.801}.
\end{assumption}

We aim to show $T^\ast=\infty$, this is achieved through the bootstrap argument. To be specific, we will prove the following refined estimates hold on $[s_0, T^\ast]$ by choosing proper $\nu$.

\begin{proposition}\label{prop:2.9}
Let $\alpha\in [3/2,2)$ and $n_{\neq}$ be the solution of \eqref{eq:2.4} with initial data $n_0$ and satisfy the Assumption \ref{ass:2.8}. Then there exist $\nu_0=\nu(n_0,\alpha)$, such that $\nu<\nu_0$, one has
\begin{itemize}
\item [(P-1)] Nonzero mode $L^2\dot{H}^{\alpha/2}$ estimate: for any $s_0\leq s\leq t\leq T^\ast$
$$
\nu\int_{s}^{t}\left\|\Lambda^{\alpha/2} n_{\neq}\right\|^2_{L^2}d\tau\leq 8C_{\neq}e^{-\lambda_{\nu,\alpha} (s-s_0)}\big\|n_{0}\big\|^2_{L^2},
$$
\item [(P-2)]Nonzero mode enhanced dissipation estimate: for any $s_0\leq t\leq T^\ast$
$$
\big\|n_{\neq}(t)\big\|^2_{L^2}\leq 2C_{\neq}e^{-\lambda_{\nu,\alpha} (t-s_0)}\big\|n_{0}\big\|^2_{L^2},
$$
\end{itemize}
where $\lambda_{\nu,\alpha}$ is defined in \eqref{eq:1.10} and $s_0$ is defined in \eqref{eq:2.801}.
\end{proposition}

\begin{remark}
In the proof of Proposition \ref{prop:2.9}, the estimate of term $\nu\partial_yn_{\neq}\mathbf{B}_1(n^0)$ to equation \eqref{eq:2.4} is difficult. We need to obtain the enhanced dissipation decay this term by Assumption \ref{ass:2.8},  the $\partial_yn_{\neq}$ contains a first-order derivative, while the $\Lambda^{\alpha/2}n_{\neq}$ only has $\alpha/2$-order derivative in Assumption \ref{ass:2.8}, which from the fractional Laplacian. This is a technical obstacle in the case of $\alpha<2$. In this paper, we develop  some new technical to overcome the difficult of the low regularity for fractional Laplacian, the details can see Appendix \ref{App:A.2}.
\end{remark}

\begin{remark}\label{rem:2.14}
Combining Assumption \ref{ass:2.8} and Proposition \ref{prop:2.9}, we know that the $n_{\neq}$  has enhanced dissipation after time $s_0$, the main reason is that we need the following estimate in the proof
$$
\left[\nu\int_{s}^{t+s}\left(\int_{\tau}^{t+s}\left\|\Lambda^{\alpha/2}\mathcal{S}_{\tau_0}P_{\neq}\right\|d\tau_0
\right)^2d\tau\right]^{1/2}
\lesssim \left(t\lambda_{\nu,\alpha}+1\right)^{1/2},
$$
see \eqref{eq:4.4401}, and it holds in the case of $s\geq s_0$. I believe that this is technical.
\end{remark}

\begin{remark}\label{rem:2.14}
We know that the time $T^\ast$ in Assumption \ref{ass:2.8} is large than $s_0+8\lambda^{-1}_{\nu}$ by local estimate, the details can see Lemma \ref{lem:A.301}. The $L^\infty L^2$ estimate of $n^0$, $L^\infty L^\infty$ estimate of $n$ and $L^\infty \dot{H}^{\alpha/2}$ estimate of $n^0$ can be obtained by the Assumption \ref{ass:2.8}, see Section \ref{sec.4}. Combining the Assumption \ref{ass:2.8} and Proposition \ref{prop:2.9}, we imply that the $T^\ast$ is infinity by bootstrap argument and the global $L^2$ estimate of $n$ is established. Based on the local solution and uniform $L^2$ estimate of $n$, the global classical solution of equation \eqref{eq:1.11} can be established by Proposition \ref{prop:2.6}, we can finish the proof of Theorem \ref{thm:1.2}. In this paper, the Proposition \ref{prop:2.9} is the most important and we can prove it by enhanced dissipation of shear flow, the details of proof can be seen in Section \ref{sec.4}.
\end{remark}

\section{Enhanced dissipation of shear flow in fractional diffusion}\label{sec.3}

In this section, we study the enhanced dissipation of shear flow in \eqref{eq:1.6} and finish the proof of Theorem \ref{thm:1.1}. Here we consider the equation \eqref{eq:1.7} and operator $\mathcal{L}_{\nu,\alpha}$ in \eqref{eq:1.8}. In addition, the shear flow $u(y)$ satisfies the assumption in Theorem \ref{thm:1.1}, the specific details are as follows

\begin{assumption}\label{ass:3.1}
Let $u(y)\in C^{\infty}(\mathbb{T})$, there exist constants $m, N\in \mathbb{N}$, $c_1>0$, $\delta_0>0$ with the property that, for
any $\lambda \in \mathbb{R}$ and $\delta\in(0,\delta_0)$, there exist finitely many points $y_1,\ldots y_n\in \mathbb{T}$ with $n\leq N$, such that
$$
|u(y)-\lambda|\geq c_1 \delta^m, \qquad \forall \  |y-y_j|\geq \delta, \quad \forall j\in \{1,\ldots n\}.
$$
\end{assumption}

Next,we prove the Theorem \ref{thm:1.1}. If $g(t,x,y)$ is the solution of equation \eqref{eq:1.7}, taking the Fourier transform in $x$, it is as follows
\begin{equation}\label{eq:3.0}
g_{k}(t,y)=\frac{1}{2\pi}\int_{\mathbb{T}}g(t,x,y)e^{-ik x}dx,\ \ \ k\neq0,
\end{equation}
then $g_{k}(t,y)$ satisfy
\begin{equation}\label{eq:3.1}
\partial_tg_k+\mathcal{L}_{\nu,\alpha,k}g_k=0,\ \ \ g_k(0,y)=g_{0,k}(y),
\end{equation}
and
\begin{equation}\label{eq:3.2}
\mathcal{L}_{\nu,\alpha,k}=\nu(k^2-\partial_{yy})^{\alpha/2}+iku(y),
\end{equation}
where $\alpha\in (0,2), k\ne0$. We study the semigroup with the operator $-\mathcal{L}_{\nu,\alpha,k}$ as generator and establish a lower bound of $\Psi(\mathcal{L}_{\nu,\alpha,k})$, which is defined by \eqref{eq:2.1} and \eqref{eq:3.2}, it is as follows
\begin{equation}\label{eq:3.3}
\Psi(\mathcal{L}_{\nu,\alpha,k})=\inf\{\|(\mathcal{L}_{\nu,\alpha,k}-i\lambda I)f\|: f\in D(\mathcal{L}_{\nu,\alpha,k}), \lambda\in \mathbb{R}, \|f\|=1 \}.
\end{equation}
We definite the operator $(k^2-\partial_{yy})^{\alpha/2}$  by Fourier series that
$$
(k^2-\partial_{yy})^{\alpha/2}f(y)=\sum_{l\in \mathbb{Z}}(k^2+l^2)^{\alpha/2}\hat{f}(l)e^{ily},\ \ \ \hat{f}(l)=\frac{1}{2\pi}\int_{\mathbb{T}}f(y)e^{-il y}dy.
$$
Since
\begin{equation}\label{eq:3.4}
\frac{1}{2}\left((l^2)^{\alpha/2}+(k^2)^{\alpha/2}\right)
\leq\left(k^2+l^2\right)^{\alpha/2}
\leq2\left((l^2)^{\alpha/2}+(k^2)^{\alpha/2}\right),
\end{equation}
one get by \eqref{eq:3.4} and Fourier series that
\begin{equation}\label{eq:3.5}
\begin{aligned}
(k^2-\partial_{yy})^{\alpha/2}f(y)&=\sum_{l\in \mathbb{Z}}(k^2+l^2)^{\alpha/2}\hat{f}(l)e^{ily}\\
&\sim\sum_{l\in \mathbb{Z}}(k^2)^{\alpha/2}\hat{f}(l)e^{ily}+\sum_{l\in \mathbb{Z}}(l^2)^{\alpha/2}\hat{f}(l)e^{ily}\\
&=|k|^\alpha f(y)+(-\partial_{yy})^{\alpha/2}f(y).
\end{aligned}
\end{equation}
Similar to \cite{Wei.2021}, we know that for any $f\in D(\mathcal{L}_{\nu,\alpha,k})=H^{\alpha}(\mathbb{T})$, we deduce by \eqref{eq:3.4} and \eqref{eq:3.5} that
$$
\begin{aligned}
\Re\langle \mathcal{L}_{\nu,\alpha,k} f, f\rangle&=\Re\left\langle \nu(k^2-\partial_{yy})^{\alpha/2}f+iku(y)f ,f\right\rangle\\
&\sim\Re\left\langle \nu|k|^\alpha f+\nu(-\partial_{yy})^{\alpha/2}f,f\right\rangle\\
&=\nu|k|^\alpha\big\|f\big\|^2_{L^2}+\nu\big\|\Lambda^{\alpha/2}_yf\big\|^2_{L^2}\geq 0,
\end{aligned}
$$
and
$$
\begin{aligned}
(\Re\lambda)\|f\|_{L^2}^2&\leq \frac{1}{2}\left(\nu|k|^\alpha\big\|f\big\|^2_{L^2}+\nu\big\|\Lambda^{\alpha/2}_yf\big\|^2_{L^2}\right)+(\Re\lambda)\big\|f\big\|_{L^2}^2\\
&\leq\Re\big\langle \mathcal{L}_{\nu,\alpha,k} f, f\big\rangle+\Re\big\langle\lambda I f,f\big\rangle\\
&=Re\big\langle (\mathcal{L}_{\nu,\alpha,k}+\lambda I)f, f\big\rangle\\
&\leq\big\|(\mathcal{L}_{\nu,\alpha,k}+\lambda I)f\big\|_{L^2}\big\|f\big\|_{L^2},
\end{aligned}
$$
therefore, the $\mathcal{L}_{\nu,\alpha,k}$ is $m$-accretive operator. Based on the Assumption \ref{ass:3.1}, we have the following Gearchart-Pr\"{u}ss type theorem and the lower bound of $\Psi(\mathcal{L}_{\nu,\alpha,k})$ in \eqref{eq:3.3}.

\begin{lemma}\label{lem:3.2}
Let $\mathcal{L}_{\nu,k}$ be an $m$-accretive operator in \eqref{eq:3.2}, then one has
\begin{equation}\label{eq:3.6}
\left\|e^{-t\mathcal{L}_{\nu,\alpha,k}}\right\|_{L^2\rightarrow L^2}\leq e^{-t\Psi(\mathcal{L}_{\nu,\alpha,k})+\pi/2}.
\end{equation}
If the shear flow $u(y)$ satisfy the Assumption \ref{ass:3.1}, $k\neq  0$ and $\nu|k|^{-1}<1$. Then there exists a positive constant $\epsilon_0$ independent of $\nu$ and $k$, such that
\begin{equation}\label{eq:3.7}
\Psi(\mathcal{L}_{\nu,\alpha,k})\geq \epsilon_0 \nu^{\frac{m}{m+\alpha}}|k|^{\frac{\alpha}{m+\alpha}},
\end{equation}
where $\Psi(\mathcal{L}_{\nu,\alpha,k})$ is defined in \eqref{eq:3.3} and $\alpha\in (0,2)$
\end{lemma}

\begin{proof}
Since $\mathcal{L}_{\nu,\alpha,k}$ is an $m$-accretive operator, the estimate of \eqref{eq:3.6} is trivial by Lemma \ref{lem:2.1}. Here we only need to prove \eqref{eq:3.7}. For any fixed $\lambda\in \mathbb{R}$, define
\begin{equation}\label{eq:3.8}
\mathcal{\widetilde{L}}_{\nu,\alpha,k}=\mathcal{L}_{\nu,\alpha,k}-i\lambda I =\nu(k^2-\partial_{yy})^{\alpha/2}+ik\left(u(y)-\widetilde{\lambda}\right),
\end{equation}
where $\widetilde{\lambda}=\lambda/ k$. Taking the set as follows
\begin{equation}\label{eq:3.9}
E=\{y\in \mathbb{T}: |y-y_j|\geq \delta, \quad \forall j\in\{1,\ldots, n\} \},
\end{equation}
and the $y_j$ satisfy the Assumption \ref{ass:3.1}. Since $u(y)$ is a continuous function, we define the function
$$
\chi: \mathbb{T}\rightarrow [-1,1]
$$
as a smooth approximation of $sign\left(u(y)-\widetilde{\lambda} \right)$, and there exists a constant $c_2>0$, such that for any $y\in \mathbb{T}$, one has
\begin{equation}\label{eq:3.10}
|\chi'(y)|\leq c_2\delta^{-1},  \ \ \ |\chi''(y)|\leq c_2\delta^{-2},
\end{equation}
and
\begin{equation}\label{eq:3.11}
\chi(y) \left(u(y)-\widetilde{\lambda} \right)\geq 0.
\end{equation}
In addition, for any $y\in E$, one has
\begin{equation}\label{eq:3.12}
\chi(y)\left(u(y)-\widetilde{\lambda} \right)=\left|u(y)-\widetilde{\lambda}\right| .
\end{equation}
For $f\in D(\mathcal{\widetilde{L}}_{\nu,\alpha,k})$ and $\|f\|_{L^2}=1$, we obtain by the definition of $\mathcal{\widetilde{L}}_{\nu,\alpha,k}$ in \eqref{eq:3.8} that
\begin{equation}\label{eq:3.13}
\left\langle\mathcal{\widetilde{L}}_{\nu,\alpha,k}f, \chi f \right\rangle =\nu\left\langle(k^2-\partial_{yy})^{\alpha/2}f,  \chi f\right\rangle+ik\left\langle \left(u(y)-\widetilde{\lambda}\right)f,\chi f\right\rangle,
\end{equation}
and consider the imaginary part of \eqref{eq:3.13}, one has
$$
\Im\left\langle\mathcal{\widetilde{L}}_{\nu,\alpha,k}f, \chi f \right\rangle =\nu \Im\left\langle\left(k^2-\partial_{yy}\right)^{\alpha/2}f,  \chi f\right\rangle+k\left\langle \left(u(y)-\widetilde{\lambda}\right)f,\chi f\right\rangle.
$$
Since $u(y)$ satisfy the Assumption \ref{ass:3.1}, one has
$$
\left|u(y)-\widetilde{\lambda}\right|\geq c_1 \delta^{m}, \ \  y\in E,
$$
then we deduce by \eqref{eq:3.11} and \eqref{eq:3.12} that
\begin{equation}\label{eq:3.14}
\left\langle \left(u(y)-\widetilde{\lambda}\right)f,\chi f\right\rangle\geq \int_{E}\left|u(y)-\widetilde{\lambda}\right|\left|f(y)\right|^2dy\geq c_1 \delta^{m}\int_{E}\left|f(y)\right|^2dy.
\end{equation}
Thus we know $\langle (u(y)-\widetilde{\lambda})f,\chi f\rangle$ is nonnegative, and we can easily get
\begin{equation}\label{eq:3.15}
\begin{aligned}
|k|\left\langle \left(u(y)-\widetilde{\lambda}\right)f,\chi f\right\rangle&=\left|k\left\langle \left(u(y)-\widetilde{\lambda}\right)f,\chi f\right\rangle\right|\\
&=\left|\Im\left\langle\mathcal{\widetilde{L}}_{\nu,\alpha,k}f, \chi f \right\rangle-\nu \Im\left\langle\left(k^2-\partial_{yy}\right)^{\alpha/2}f,  \chi f\right\rangle\right|.
\end{aligned}
\end{equation}
Combining $\langle|k|^{\alpha}f,  \chi f\rangle$ is real and \eqref{eq:3.5}, we have
$$
\Im\left\langle\left(k^2-\partial_{yy}\right)^{\alpha/2}f,  \chi f\right\rangle\sim \Im\left\langle(-\partial_{yy})^{\alpha/2}f,  \chi f\right\rangle.
$$
By Cauchy-Schwartz inequality and Lemma \ref{lem:2.2}, one has
$$
\begin{aligned}
\left|\left\langle(-\partial_{yy})^{\alpha/2}f, \chi f\right\rangle\right|&=\left|\left\langle\Lambda_{y}^{\alpha/2}f,  \Lambda_{y}^{\alpha/2}(\chi f)\right\rangle\right|\\
&\lesssim\left\|\Lambda_{y}^{\alpha/2}f\right\|_{L^2}
\left\|\Lambda_{y}^{\alpha/2}(\chi f)\right\|_{L^2},
\end{aligned}
$$
and by Lemma \ref{lem:2.4} and \ref{lem:2.01}, we obtain
$$
\begin{aligned}
\left\|\Lambda_{y}^{\alpha/2}(\chi f)\right\|_{L^2}
&\lesssim \left\|\Lambda_{y}^{\alpha/2}\chi \right\|_{L^\infty}\big\|f\big\|_{L^2}+\big\|\chi\big\|_{L^\infty}\left\|\Lambda_{y}^{\alpha/2} f\right\|_{L^2}\\
&\lesssim c_2\delta^{-\alpha/2}\big\|f\big\|_{L^2}+\left\|\Lambda_{y}^{\alpha/2}f\right\|_{L^2}.
\end{aligned}
$$
Then we imply from \eqref{eq:3.15} that
\begin{equation}\label{eq:3.16}
\begin{aligned}
|k|\left\langle \left(u(y)-\widetilde{\lambda}\right)f,\chi f\right\rangle&=\left|\Im\left\langle\mathcal{\widetilde{L}}_{\nu,\alpha,k}f, \chi f \right\rangle-\nu \Im\left\langle\left(k^2-\partial_{yy}\right)^{\alpha/2}f,  \chi f\right\rangle\right|\\
&\lesssim \left\|\mathcal{\widetilde{L}}_{\nu,\alpha,k}f\right\|_{L^2}\big\|f\big\|_{L^2}+\nu\left\|\Lambda_{y}^{\alpha/2}f\right\|^2_{L^2}
+ c_2\nu\delta^{-\alpha/2}\big\|f\big\|_{L^2}\left\|\Lambda_{y}^{\alpha/2} f\right\|_{L^2}.
\end{aligned}
\end{equation}
Combining \eqref{eq:3.14} and \eqref{eq:3.16}, we obtain
$$
\begin{aligned}
\int_{E}\left|f(y)\right|^2dy&\leq c^{-1}_1\delta^{-m}\left\langle \left(u(y)-\widetilde{\lambda}\right)f,\chi f\right\rangle\\
&\lesssim c^{-1}_1\delta^{-m}|k|^{-1}\left(\left\|\mathcal{\widetilde{L}}_{\nu,\alpha,k}f\right\|_{L^2}\big\|f\big\|_{L^2}
+\nu\left\|\Lambda_{y}^{\alpha/2}f\right\|^2_{L^2}
+ c_2\nu\delta^{-\alpha/2}\big\|f\big\|_{L^2}\left\|\Lambda_{y}^{\alpha/2} f\right\|_{L^2}\right).\!\!\!\!\!
\end{aligned}
$$
Since
$$
\left\langle\mathcal{\widetilde{L}}_{\nu,\alpha,k}f,  f \right\rangle =\nu\left\langle(k^2-\partial_{yy})^{\alpha/2}f,  f\right\rangle+ik\left\langle (u(y)-\widetilde{\lambda})f,f\right\rangle,
$$
and combining \eqref{eq:3.5} and Lemma \ref{lem:2.2}, one gets
$$
\Re\left\langle\mathcal{\widetilde{L}}_{\nu,\alpha,k}f, f \right\rangle =\nu\left\langle(k^2-\partial_{yy})^{\alpha/2}f,  f\right\rangle\sim \nu|k|^{\alpha}\big\|f\big\|^2_{L^2}+\nu\left\|\Lambda_{y}^{\alpha/2}f\right\|^2_{L^2}.
$$
Then we deduce by Cauchy-Schwartz inequality that
\begin{equation}\label{eq:3.1701}
\left\|\Lambda_{y}^{\alpha/2}f\right\|_{L^2}\lesssim\nu^{-1/2}\left({\rm Re}\left\langle\mathcal{\widetilde{L}}_{\nu,\alpha,k}f,  f \right\rangle \right)^{1/2}\lesssim\nu^{-\frac{1}{2}}\left\|\mathcal{\widetilde{L}}_{\nu,\alpha,k}f\right\|^{1/2}_{L^2}\big\|f\big\|^{1/2}_{L^2},
\end{equation}
and one gets
\begin{equation}\label{eq:3.17}
\int_{E}\left|f(y)\right|^2dy
\leq Cc^{-1}_1\delta^{-m}|k|^{-1}\left(\left\|\mathcal{\widetilde{L}}_{\nu,\alpha,k}f\right\|_{L^2}\big\|f\big\|_{L^2}
+c_2\nu^{1/2} \delta^{-\alpha/2}\left\|\mathcal{\widetilde{L}}_{\nu,\alpha,k}f\right\|^{1/2}_{L^2}\big\|f\big\|^{3/2}_{L^2}\right).
\end{equation}
Since
$$
\begin{aligned}
&Cc^{-1}_1\delta^{-m}|k|^{-1}c_2\nu^{1/2} \delta^{-\alpha/2}\left\|\mathcal{\widetilde{L}}_{\nu,\alpha,k}f\right\|^{1/2}_{L^2}\big\|f\big\|^{3/2}_{L^2}\\
\leq&\frac{1}{4}\|f\|^{2}_{L^2}+C^2c^{-2}_1\delta^{-2m-\alpha}|k|^{-2}c^2_2\nu \left\|\mathcal{\widetilde{L}}_{\nu,\alpha,k}f\right\|_{L^2}\big\|f\big\|_{L^2},
\end{aligned}
$$
there exists a constant $\widetilde{C}=C(c_1,c_2)$, such that the \eqref{eq:3.17} can be written as
\begin{equation}\label{eq:3.19}
\int_{E}\left|f(y)\right|^2dy
\leq \widetilde{C}\left(\delta^{-m}|k|^{-1}+|k|^{-2}\nu \delta^{-2m-\alpha} \right)\left\|\mathcal{\widetilde{L}}_{\nu,\alpha,k}f\right\|_{L^2}\big\|f\big\|_{L^2}+\frac{1}{4}\big\|f\big\|^{2}_{L^2}.
\end{equation}
Denoted $E^c$ as the complement of $E$, then $|E^c|\leq 2N\delta$ by the definition of $E$ in \eqref{eq:3.9}. For any $\alpha\in (0,2)$, let $p=4/(2-\alpha),q=4/\alpha $, one has 
$$
\int_{E^c}\left|f(y)\right|^2dy\lesssim \left(\int_{E^c}1^qdy\right)^{\frac{2}{q}}\left(\int_{E^c}\left|f(y)\right|^pdy\right)^{\frac{2}{p}}
\lesssim (2N\delta)^{\alpha/2}\big\|f\big\|^2_{L^p},
$$
and 
$$
\big\|f\big\|_{L^p}\lesssim\big\|f\big\|^{1/2}_{L^2}\left\|\Lambda_y^{\alpha/2}f\right\|^{1/2}_{L^2}.
$$
Then we deduce by \eqref{eq:3.1701} that 
\begin{equation}\label{eq:3.20}
\begin{aligned}
\int_{E^c}\left|f(y)\right|^2dy
&\lesssim\left(2N\delta\right)^{\alpha/2}\|f\|_{L^2}\left\|\Lambda_y^{\alpha/2}f\right\|_{L^2}\\
&\lesssim\left(2N\delta\right)^{\alpha/2}\nu^{-1/2}\left\|\mathcal{\widetilde{L}}_{\nu,\alpha,k}f\right\|^{1/2}_{L^2}\big\|f\big\|^{3/2}_{L^2}\\
&\leq \frac{1}{4}\big\|f\big\|^2_{L^2}+ C(N\delta)^\alpha\nu^{-1}\left\|\mathcal{\widetilde{L}}_{\nu,\alpha,k}f\right\|_{L^2}\big\|f\big\|_{L^2}.
\end{aligned}
\end{equation}
Combining \eqref{eq:3.19} and \eqref{eq:3.20}, we have
\begin{equation}\label{eq:3.21}
\big\|f\big\|^2_{L^2}\leq 2\left(\widetilde{C}\delta^{-m}|k|^{-1}+\widetilde{C}|k|^{-2}\nu \delta^{-2m-\alpha}+C(N\delta)^\alpha \nu^{-1} \right)\left\|\mathcal{\widetilde{L}}_{\nu,\alpha,k}f\right\|_{L^2}\big\|f\big\|_{L^2}.
\end{equation}
Taking $\delta$ small enough and
$$
\delta=c_3\left(\nu/|k|\right)^{\frac{1}{m+\alpha}},
$$
where $c_3>0$ is a small constant. Then there exists a constant $C_0=C(c_1,c_2,c_3,\alpha,m,N)$, such that
$$
\widetilde{C}\delta^{-m}|k|^{-1}+\widetilde{C}|k|^{-2}\nu \delta^{-2m-\alpha}+C(N\delta)^\alpha \nu^{-1} \leq C_0\nu^{-\frac{m}{m+\alpha}}|k|^{-\frac{\alpha}{m+\alpha}}.
$$
Therefore, we can imply by \eqref{eq:3.21} that
$$
\left\|\mathcal{\widetilde{L}}_{\nu,\alpha,k}f\right\|_{L^2}\geq \epsilon_0\nu^{\frac{m}{m+\alpha}}|k|^{\frac{\alpha}{m+\alpha}}\big\|f\big\|_{L^2},
$$
where $\epsilon_0=1/2C_0$. Since $f$ is arbitrary, we deduce by the definition of $\Psi(\mathcal{L}_{\nu,\alpha,k})$ in  \eqref{eq:3.3} that
$$
\Psi(\mathcal{L}_{\nu,\alpha,k})\geq \epsilon_0\nu^{\frac{m}{m+\alpha}}|k|^{\frac{\alpha}{m+\alpha}}.
$$
This completes the proof of Lemma \ref{lem:3.2}.
\end{proof}

Next, we give the proof of Theorem \ref{thm:1.1} based on the Lemma \ref{lem:3.2}.

\begin{proof}[The proof of Theorem \ref{thm:1.1}]
For any $h(x,y)\in L^2(\mathbb{T}^2)$, we consider the equation \eqref{eq:1.7} with initial data $g_0(x,y)=P_{\neq}h(x,y)$, then the solution can be written as
$$
g(t,x,y)=e^{-t\mathcal{L}_{\nu,\alpha}}g_0=e^{-t\mathcal{L}_{\nu,\alpha}}P_{\neq}h,
$$
and we can easily get $P_0g=0$. Here the operator $\mathcal{L}_{\nu,\alpha}$ is defined in \eqref{eq:1.8}, the operator $P_0$ and $P_{\neq}$ are defined in \eqref{eq:1.9}. Through the Fourier series,  one has 
\begin{equation}\label{eq:3.23}
g(t,x,y)=e^{-t\mathcal{L}_{\nu,\alpha}}P_{\neq}h=\sum_{k\neq0}g_k(t,y)e^{ikx},
\end{equation}
where $g_k(t,y)$ is defined in \eqref{eq:3.0} and is the solution of equation \eqref{eq:3.1}. Then we have
$$
g_{k}(t,y)=e^{-t\mathcal{L}_{\nu,\alpha,k}}g_{0,k},
$$
and the operator $\mathcal{L}_{\nu,\alpha,k}$ is defined in \eqref{eq:3.2}. By Lemma \ref{lem:3.2}, one gets
$$
\left\|e^{-t\mathcal{L}_{\nu,\alpha,k}}\right\|_{L^2\rightarrow L^2}\leq e^{-\lambda_{\nu,\alpha,k}t+\pi/2},
$$
where $\lambda_{\nu,\alpha,k}=\epsilon_0\nu^{\frac{m}{m+\alpha}}|k|^{\frac{\alpha}{m+\alpha}}$. Then we deduce by Plancherel equality and \eqref{eq:3.23} that
$$
\begin{aligned}
\big\|g\big\|^2_{L^2}&=\big\|e^{-t\mathcal{L}_{\nu,\alpha}}P_{\neq}h\big\|^2_{L^2}=\sum_{k\neq 0}\big\|g_k\big\|^2_{L^2}=\sum_{k\neq 0}\left\|e^{-t\mathcal{L}_{\nu,\alpha,k}}g_{0,k}\right\|^2_{L^2}\\
&\leq \sum_{k\neq 0}e^{-2\lambda_{\nu,\alpha,k} t+\pi}\big\|g_{0,k}\big\|^2_{L^2}\leq \max_{k\neq 0}e^{-2\lambda_{\nu,\alpha,k}t+\pi}\sum_{k\neq 0}\big\|g_{0,k}\big\|^2_{L^2}\\
&\leq e^{-2\lambda'_{\nu,\alpha} t+\pi}\big\|g_{0}\big\|^2_{L^2}\leq e^{-2\lambda'_{\nu,\alpha} t+\pi}\big\|h\big\|^2_{L^2}.
\end{aligned}
$$
Since $h$ is arbitrary, we have
$$
\left\|e^{-t\mathcal{L}_{\nu,\alpha}}P_{\neq}\right\|_{L^2\rightarrow L^2}\leq e^{-\lambda'_{\nu,\alpha }t+\pi/2}.
$$
This completes the proof of Theorem \ref{thm:1.1}.
\end{proof}

\section{Suppression of blow-up and global $L^2$ estimate}\label{sec.4}

In this section, we will prove the  Theorem \ref{thm:1.2}. We know from the Remark \ref{rem:2.14} that we only need to show the Proposition \ref{prop:2.9}. Some estimates are used in the proof. First, we give the following lemma.

\begin{lemma}[$L^\infty L^2$ estimate of $n^0$]\label{lem:4.1}
Let $\alpha\in [3/2,2)$,  the $n^0$ and $n_{\neq}$ are the solutions of \eqref{eq:2.3} and \eqref{eq:2.4} with initial data $n_0$. If $n_{\neq}$ satisfy the Assumption \ref{ass:2.8}, then there exist $\nu_0=\nu(n_0,\alpha)$ and a positive constant $B_1=B(\|n_0\|_{L^2}, M, C_{\neq},\alpha)$, if $\nu<\nu_0$, we have
$$
\big\|n^{0}\big\|_{L^\infty(0, T^\ast;L^2)}\leq B_1.
$$
\end{lemma}

\begin{proof}
Let us multiply both sides of \eqref{eq:2.3} by $n^0$ and integrate it over $\mathbb{T}$, one has
\begin{equation}\label{eq:4.1}
\frac{1}{2}\frac{d}{dt}\big\|n^0\big\|^2_{L^2}+\nu\left\|\Lambda_y^{\alpha/2} n^0\right\|^2_{L^2}+\nu \int_{\mathbb{T}}\partial_y(n^{0}\mathbf{ B}_1(n^0))n^0dy+\nu\int_{\mathbb{T}}\left(\nabla\cdot(n_{\neq} \mathbf{B}_2(n_{\neq}))\right)^{0}n^0dy=0.
\end{equation}
Using integral by part, Lemma \ref{lem:A.1} and energy estimate, one has
$$
\left| \int_{\mathbb{T}}\partial_y(n^{0}\mathbf{ B}_1(n^0))n^0dy\right|
=\left|\frac{1}{2} \int_{\mathbb{T}}\partial_y\mathbf{ B}_1(n^0)(n^0)^2dy\right|\lesssim \big\|\partial_y\mathbf{ B}_1(n^0)\big\|_{L^2}\big\|n^0\big\|^2_{L^4},
$$
and
$$
\big\|\partial_y\mathbf{ B}_1(n^0)\big\|_{L^2}\big\|n^0\big\|^2_{L^4}\lesssim \big\|n^0\big\|^{3-\frac{1}{\alpha}}_{L^2}\left\|\Lambda^{\alpha/2}_y n^0\right\|^{\frac{1}{\alpha}}_{L^2}\leq \frac{1}{4}\left\|\Lambda^{\alpha/2}_y n^0\right\|^{2}_{L^2}+C\big\|n^0\big\|^{3+\frac{1}{2\alpha-1}}_{L^2}.
$$
Then the third term of \eqref{eq:4.1} is estimated as
\begin{equation}\label{eq:4.2}
\left|\nu\int_{\mathbb{T}}\partial_y(n^{0}\mathbf{ B}_1(n^0))n^0dy\right|
\leq \frac{\nu}{4}\left\|\Lambda^{\alpha/2}_y n^0\right\|^{2}_{L^2}+C\nu\big\|n^0\big\|^{3+\frac{1}{2\alpha-1}}_{L^2}.
\end{equation}
Since
$$
\begin{aligned}
\int_{\mathbb{T}}\left(\nabla\cdot(n_{\neq} \mathbf{B}_2(n_{\neq}))\right)^{0}n^0dy
&=\frac{1}{2\pi}\int_{\mathbb{T}^2}\nabla\cdot(n_{\neq} \mathbf{B}_2(n_{\neq}))n^0dxdy\\
&=\frac{1}{2\pi}\int_{\mathbb{T}^2}\partial_y(n_{\neq}\partial_y\nabla^{-1}\cdot\mathbf{B}_2(n_{\neq}))n^0dxdy,
\end{aligned}
$$
where $\nabla^{-1}=(\partial_x^{-1},\partial_y^{-1})$, we deduce by Lemma \ref{lem:2.4} and \ref{lem:2.5} that
$$
\left|\frac{1}{2\pi}\int_{\mathbb{T}^2}\partial_y(n_{\neq}\partial_y\nabla^{-1}\cdot\mathbf{B}_2(n_{\neq}))n^0dxdy\right|
\lesssim\left\|\Lambda^{-\alpha/2}\partial_y(n_{\neq}\partial_y\nabla^{-1}\cdot\mathbf{B}_2(n_{\neq}))\right\|_{L^2}
\left\|\Lambda^{\alpha/2}_y n^0\right\|_{L^2},
$$
and 
$$
\begin{aligned}
&\left\|\Lambda^{-\alpha/2}\partial_y(n_{\neq}\partial_y\nabla^{-1}\cdot\mathbf{B}_2(n_{\neq}))\right\|_{L^2}
\lesssim\left\|\Lambda^{1-\alpha/2}(n_{\neq}\partial_y\nabla^{-1}\cdot\mathbf{B}_2(n_{\neq}))\right\|_{L^2}\\
\lesssim&\left\|\Lambda^{1-\alpha/2}n_{\neq}\right\|_{L^2}
\left\|\partial_y\nabla^{-1}\cdot\mathbf{B}_2(n_{\neq})\right\|_{L^\infty}
+\big\|n_{\neq}\big\|_{L^2}
\left\|\Lambda^{1-\alpha/2}\partial_y\nabla^{-1}\cdot\mathbf{B}_2(n_{\neq})\right\|_{L^\infty}.
\end{aligned}
$$
Combining $\alpha\geq 3/2$ and Lemma \ref{lem:A.1}, one has
$$
\begin{aligned}
&\left\|\Lambda^{1-\alpha/2}n_{\neq}\right\|_{L^2}
\lesssim\left\|\Lambda^{\alpha/2}n_{\neq}\right\|^{2/\alpha-1}_{L^2}\big\|n_{\neq}\big\|^{2-2/\alpha}_{L^2},\\
&\left\|\Lambda^{1-\alpha/2}\partial_y\nabla^{-1}\cdot\mathbf{B}_2(n_{\neq})\right\|_{L^\infty}
\lesssim\big\|\partial_y\nabla^{-1}\cdot\mathbf{B}_2(n_{\neq})\big\|^{\alpha-1}_{L^\infty}
\big\|\partial_y\mathbf{B}_2(n_{\neq})\big\|^{2-\alpha}_{L^4}\lesssim \big\|n_{\neq}\big\|_{L^4},\\
&\big\|n_{\neq}\big\|_{L^4}\lesssim  \left\|\Lambda^{\alpha/2} n_{\neq}\right\|^{1/\alpha}_{L^2}\big\|n_{\neq}\big\|^{1-1/\alpha}_{L^2},
\end{aligned}
$$
and
\begin{equation}\label{eq:4.201}
\big\|\partial_y\nabla^{-1}\cdot\mathbf{B}_2(n_{\neq})\big\|_{L^\infty}
\lesssim\big\|\partial_y\mathbf{B}_2(n_{\neq})\big\|_{L^4}\lesssim \big\|n_{\neq}\big\|_{L^4}.
\end{equation}
Therefore, we can easily get
$$
\begin{aligned}
&\left\|\Lambda^{-\alpha/2}\partial_y(n_{\neq}\partial_y\nabla^{-1}\cdot\mathbf{B}_2(n_{\neq}))\right\|_{L^2}
\left\|\Lambda^{\alpha/2}_y n^0\right\|_{L^2}\\
\lesssim & \left\|\Lambda^{\alpha/2} n_{\neq}\right\|^{3/\alpha-1}_{L^2}\big\|n_{\neq}\big\|^{3-3/\alpha}_{L^2}
\left\|\Lambda^{\alpha/2}_y n^0\right\|_{L^2}
+\left\|\Lambda^{\alpha/2} n_{\neq}\right\|^{1/\alpha}_{L^2}\big\|n_{\neq}\big\|^{2-1/\alpha}_{L^2}\left\|\Lambda^{\alpha/2}_y n^0\right\|_{L^2}.
\end{aligned}
$$
If $\alpha>3/2$, we imply
\begin{equation}\label{eq:4.3}
\begin{aligned}
&\left\|\Lambda^{\alpha/2} n_{\neq}\right\|^{3/\alpha-1}_{L^2}\left\|n_{\neq}\right\|^{3-3/\alpha}_{L^2}
\left\|\Lambda^{\alpha/2}_y n^0\right\|_{L^2}\\
\leq& \frac{1}{8}\left\|\Lambda^{\alpha/2}_y n^0\right\|^2_{L^2}+C\left\|\Lambda^{\alpha/2} n_{\neq}\right\|^{6/\alpha-2}_{L^2}\big\|n_{\neq}\big\|^{6-6/\alpha}_{L^2}\\
\leq&\frac{1}{8}\left\|\Lambda^{\alpha/2}_y n^0\right\|^2_{L^2}+\frac{1}{128C_{\neq}}\left\|\Lambda^{\alpha/2} n_{\neq}\right\|^{2}_{L^2}+C\big\| n_{\neq}\big\|^{\frac{6\alpha-6}{2\alpha-3}}_{L^2},
\end{aligned}
\end{equation}
and
\begin{equation}\label{eq:4.4}
\begin{aligned}
&\left\|\Lambda^{\alpha/2} n_{\neq}\right\|^{1/\alpha}_{L^2}\left\|n_{\neq}\right\|^{2-1/\alpha}_{L^2}\left\|\Lambda^{\alpha/2}_y n^0\right\|_{L^2}\\
\leq&\frac{1}{8}\left\|\Lambda^{\alpha/2}_y n^0\right\|^2_{L^2}+C\left\|\Lambda^{\alpha/2} n_{\neq}\right\|^{2/\alpha}_{L^2}\big\|n_{\neq}\big\|^{4-2/\alpha}_{L^2}\\
\leq&\frac{1}{8}\left\|\Lambda^{\alpha/2}_y n^0\right\|^2_{L^2}+\frac{1}{128C_{\neq}}\left\|\Lambda^{\alpha/2} n_{\neq}\right\|^{2}_{L^2}+C\big\| n_{\neq}\big\|^{\frac{4\alpha-2}{\alpha-1}}_{L^2},
\end{aligned}
\end{equation}
Then the fourth term of \eqref{eq:4.1} is estimated as
\begin{equation}\label{eq:4.5}
\begin{aligned}
\left|\nu\int_{\mathbb{T}}\left(\nabla\cdot(n_{\neq} \mathbf{B}_2(n_{\neq}))\right)^{0}n^0dy\right|
\leq&\frac{\nu}{8}\left\|\Lambda^{\alpha/2}_y n^0\right\|^2_{L^2}+\frac{\nu}{64C_{\neq}}\left\|\Lambda^{\alpha/2} n_{\neq}\right\|^{2}_{L^2}\\
&+C\nu\big\| n_{\neq}\big\|^{\frac{6\alpha-6}{2\alpha-3}}_{L^2}
+C\nu\big\|n_{\neq}\|^{\frac{4\alpha-2}{\alpha-1}}_{L^2}.
\end{aligned}
\end{equation}
Combining \eqref{eq:4.1}, \eqref{eq:4.2} and \eqref{eq:4.5}, one has
\begin{equation}\label{eq:4.701}
\begin{aligned}
\frac{d}{dt}\big\|n^0\big\|^2_{L^2}+\nu\left\|\Lambda_y^{\alpha/2} n^0\right\|^2_{L^2}
\leq& C\nu\big\|n^0\big\|^{3+\frac{1}{2\alpha-1}}_{L^2}+\frac{\nu}{32C_{\neq}}\left\|\Lambda^{\alpha/2} n_{\neq}\right\|^{2}_{L^2}\\
&+C\nu\big\| n_{\neq}\big\|^{\frac{6\alpha-6}{2\alpha-3}}_{L^2}
+C\nu\big\| n_{\neq}\big\|^{\frac{4\alpha-2}{\alpha-1}}_{L^2}.
\end{aligned}
\end{equation}
By Nash inequality and  $\|n^{0}\|_{L^1}\leq CM$, one has
$$
-\left\|\Lambda_y^{\alpha/2} n^0\right\|^2_{L^2}\leq -\frac{\big\|n^0\big\|^{2\alpha+2}_{L^2}}{CM^{2\alpha}}.
$$
Then we have
\begin{equation}\label{eq:4.6}
\begin{aligned}
\frac{d}{dt}\big\|n^0\big\|^2_{L^2}
&\leq-\frac{\nu\big\|n^0\big\|^{3+\frac{1}{2\alpha-1}}_{L^2}}{CM^{2\alpha}}
\left(\big\|n^0\big\|^{2\alpha-1-\frac{1}{2\alpha-1}}_{L^2}-CM^{2\alpha}\right)\\
&+\frac{\nu}{32C_{\neq}}\left\|\Lambda^{\alpha/2} n_{\neq}\right\|^{2}_{L^2}
+C\nu\big\| n_{\neq}\big\|^{\frac{6\alpha-6}{2\alpha-3}}_{L^2}
+C\nu\big\| n_{\neq}\big\|^{\frac{4\alpha-2}{\alpha-1}}_{L^2}.
\end{aligned}
\end{equation}
Define
$$
G(t)=\int_{0}^t\frac{\nu}{32C_{\neq}}\left\|\Lambda^{\alpha/2} n_{\neq}\right\|^{2}_{L^2}
+C\nu\big\| n_{\neq}\big\|^{\frac{4\alpha-2}{\alpha-1}}_{L^2}
+C\nu\big\| n_{\neq}\big\|^{\frac{6\alpha-6}{2\alpha-3}}_{L^2}d\tau,
$$
and for any $t\in [0,T^\ast]$, we deduce by Assumption \ref{ass:2.8}, \eqref{eq:A.1601}, \eqref{eq:A.1602} and $\nu$ is small enough that
\begin{equation}\label{eq:4.7}
\begin{aligned}
G(t)&\leq\int_{0}^{s_0}\frac{\nu}{32C_{\neq}}\left\|\Lambda^{\alpha/2} n_{\neq}\right\|^{2}_{L^2}
+C\nu\big\|n_{\neq}\big\|^{\frac{4\alpha-2}{\alpha-1}}_{L^2}
+C\nu\big\|n_{\neq}\big\|^{\frac{6\alpha-6}{2\alpha-3}}_{L^2}d\tau\\
&\ \ \ +\int_{s_0}^t\frac{\nu}{32C_{\neq}}\left\|\Lambda^{\alpha/2} n_{\neq}\right\|^{2}_{L^2}
+C\nu\big\| n_{\neq}\big\|^{\frac{4\alpha-2}{\alpha-1}}_{L^2}
+C\nu\big\| n_{\neq}\big\|^{\frac{6\alpha-6}{2\alpha-3}}_{L^2}d\tau\\
&\leq K_0\big\|n_0\big\|^2_{L^2},
\end{aligned}
\end{equation}
where $K_0=K(\|n_0\|_{L^2},\alpha, C_{\neq},M)$. Then we obtain from \eqref{eq:4.6} that
\begin{equation}\label{eq:4.8}
\frac{d}{dt}\left(\big\|n^0\big\|^2_{L^2}-G(t)\right)
\leq-\frac{\nu\big\|n^0\big\|^{3+\frac{1}{2\alpha-1}}_{L^2}}{CM^{2\alpha}}
\left(\big\|n^0\big\|^{2\alpha-1-\frac{1}{2\alpha-1}}_{L^2}-G(t)-CM^{2\alpha}\right).
\end{equation}
Combining \eqref{eq:4.7} and \eqref{eq:4.8}, we imply that
$$
\big\|n^0\big\|^{2}_{L^2}\leq\left( C M^{2\alpha}+\big\|n_0\big\|^2_{L^2}\right)^{\frac{2\alpha-1}{4\alpha^2-4\alpha}}+(K_0+1)\big\|n_0\big\|^2_{L^2}\triangleq B^2_1.
$$
This completes the proof of Lemma \ref{lem:4.1}.
\end{proof}

\begin{remark}
In the \eqref{eq:4.3}, we used the Young's inequality, then
$$
6/\alpha-2<2\Rightarrow\alpha>3/2.
$$
And for $\alpha=3/2$, the \eqref{eq:4.3} can be written as
$$
\left\|\Lambda^{\alpha/2} n_{\neq}\right\|^{3/\alpha-1}_{L^2}\big\|n_{\neq}\big\|^{3-3/\alpha}_{L^2}
\left\|\Lambda^{\alpha/2}_y n^0\right\|_{L^2}
\leq \frac{1}{8}\left\|\Lambda^{\alpha/2}_y n^0\right\|^2_{L^2}+C\left\|\Lambda^{\alpha/2} n_{\neq}\right\|^{2}_{L^2}\big\|n_{\neq}\big\|^{2}_{L^2},
$$
and by Assumption \ref{ass:2.8} and \eqref{eq:A.1602}, one has
$$
\begin{aligned}
&C\nu\int_0^t\left\|\Lambda^{\alpha/2} n_{\neq}\right\|^{2}_{L^2}\big\|n_{\neq}\big\|^{2}_{L^2}d\tau\\
&\leq C\nu\int_{0}^{s_0}\left\|\Lambda^{\alpha/2} n_{\neq}\right\|^{2}_{L^2}\big\|n_{\neq}\big\|^{2}_{L^2}d\tau
+C\nu\int_{s_0}^{t}\left\|\Lambda^{\alpha/2} n_{\neq}\right\|^{2}_{L^2}\big\|n_{\neq}\big\|^{2}_{L^2}d\tau\\
&\leq C\left( \big\|n_0\big\|^{\frac{\alpha}{\alpha-1}}_{L^2}+\overline{n}\right)\big\|n_0\big\|^2_{L^2}
+4CC_{\neq}\big\|n_0\big\|^2_{L^2}\nu\int_0^t\left\|\Lambda^{\alpha/2} n_{\neq}\right\|^{2}_{L^2}d\tau\\
&\leq C\left( \big\|n_0\big\|^{\frac{\alpha}{\alpha-1}}_{L^2}+\overline{n}\right)\big\|n_0\big\|^2_{L^2}+64 CC^2_{\neq}\big\|n_0\big\|^4_{L^2}.
\end{aligned}
$$
Thus, we can also obtain the estimate of $\|n^0\|_{L^2}$. In this paper, we use energy inequality in the case of $\alpha\geq3/2$, which will not be pointed out in later sections.
\end{remark}

\begin{remark}
In the \eqref{eq:4.7}, for any $t\in [0,s_0]$, we also have
$$
G(t)\leq\int_{0}^{s_0}\frac{\nu}{32C_{\neq}}\left\|\Lambda^{\alpha/2} n_{\neq}\right\|^{2}_{L^2}
+C\nu\big\|n_{\neq}\big\|^{\frac{4\alpha-2}{\alpha-1}}_{L^2}
+C\nu\big\|n_{\neq}\big\|^{\frac{6\alpha-6}{2\alpha-3}}_{L^2}d\tau
\leq K_0\big\|n_0\big\|^2_{L^2}.
$$
\end{remark}

\vskip .05in

To improve the Assumptions \ref{ass:2.8}, we are left to complete the $L^\infty L^\infty$ estimate of $n$. This will be achieved by the nonlinear maximum principle, see Lemma \ref{lem:2.3}.

\begin{lemma}[$L^\infty L^\infty$ estimate of $n$]\label{lem:4.3}
Let $\alpha\in [3/2,2)$,  the $n, n^0$ and $n_{\neq}$ are the solutions of\eqref{eq:1.11}, \eqref{eq:2.3} and \eqref{eq:2.4} with initial data $n_0$. If $n_{\neq}$ satisfy the Assumption \ref{ass:2.8}, then there exist $\nu_0=\nu(n_0,\alpha)$ and a positive constant $B_2=B(\|n_0\|_{L^2},\|n_0\|_{L^\infty}, B_1, M, C_{\neq},\alpha)$, if $\nu<\nu_0$, we have
$$
\big\|n\big\|_{L^\infty(0, T^\ast;L^\infty)}\leq B_2.
$$
\end{lemma}

\begin{proof}
Define
$$
\mathcal{E}(t)=n(t,\overline{x}_t,\overline{y}_t)=\sup_{(x,y)\in \mathbb{T}^2}n(t,x,y).
$$
For any fixed $t\geq0$, using the vanishing of a derivation at the point of maximum, we can observe that
$$
\partial_t n(t,\overline{x}_t,\overline{y}_t)=\frac{d}{dt}\mathcal{E}(t),\ \
u(y)\partial_x n(t,\overline{x}_t,\overline{y}_t)=0,
$$
we obtain from \eqref{eq:1.3} that
$$
\nabla\cdot(n \mathbf{B}(n))(t,\overline{x}_t,\overline{y}_t)=-\mathcal{E}^2(t)+\overline{n}\mathcal{E}(t),
$$
and we denote
$$
(-\Delta)^{\alpha/2}n(t,x)\big|_{x=\overline{x}_t,y=\overline{y}_t}
=(-\Delta)^{\alpha/2}\mathcal{E}(t),
$$
Then combining \eqref{eq:1.11}, we deduce that the $\mathcal{E}=\mathcal{E}(t)$ follows
\begin{equation}\label{eq:4.9}
\frac{d}{dt}\mathcal{E}+\nu(-\Delta)^{\alpha/2}\mathcal{E}
-\nu\mathcal{E}^2+\nu\overline{n}\mathcal{E}=0.
\end{equation}
Combining Lemma \ref{lem:2.3}, Assumption \ref{ass:2.8} and Lemma \ref{lem:4.1}, we know that $\mathcal{E}(t)$ satisfies
\begin{equation}\label{eq:4.10}
\mathcal{E}(t)\lesssim \big\|n\big\|_{L^2}\lesssim (4C_{\neq})^{1/2}\big\|n_0\big\|_{L^2}+2\big\|n_0\big\|_{L^2}+B_1\triangleq K_1,
\end{equation}
or $\mathcal{E}(t)$ satisfies
\begin{equation}\label{eq:4.11}
(-\Delta)^{\alpha/2}\mathcal{E}(t)\geq C(\alpha)\frac{\mathcal{E}^{1+\alpha }}{\|n\|_{L^2}^{\alpha}}
\geq C(\alpha)\frac{\mathcal{E}^{1+\alpha }}{K_1^{\alpha}}.
\end{equation}
Combining \eqref{eq:4.9} and \eqref{eq:4.11}, one has
\begin{equation}\label{eq:4.12}
\frac{d}{dt}\mathcal{E}\leq -\nu C(\alpha)\frac{\mathcal{E}^{1+\alpha }}{K_1^{\alpha}}+\nu\mathcal{E}^2-\nu\overline{n}\mathcal{E}\leq-\nu C(\alpha)\frac{\mathcal{E}^{1+\alpha }}{K_1^{\alpha}}+\nu\mathcal{E}^2.
\end{equation}
Then we deduce by $\alpha\geq3/2$ and \eqref{eq:4.12} that
\begin{equation}\label{eq:4.13}
\mathcal{E}(t)\leq \big\|n_0\big\|_{L^\infty}+K_1^{\frac{\alpha}{\alpha-1}}C(\alpha)^{-\frac{1}{\alpha-1}}.
\end{equation}
Combining the definition of $\mathcal{E}(t)$, \eqref{eq:4.10} and \eqref{eq:4.13}, we have
$$
\big\|n\big\|_{L^\infty}\leq\|n_0\|_{L^\infty}+K_1^{\frac{\alpha}{\alpha-1}}C(\alpha)^{-\frac{1}{\alpha-1}}+K_1\triangleq B_2.
$$
This completes the proof of Lemma \ref{lem:4.3}.
\end{proof}

Next, we consider the $L^\infty \dot{H}^{\alpha-1}$ estimate of $n^0$ in the following lemma.

\begin{lemma}[$L^\infty \dot{H}^{\alpha-1}$ estimate of $n^0$]\label{lem:4.4}
Let $\alpha\in [3/2,2)$,  the $n^0$ and $n_{\neq}$ are the solutions of \eqref{eq:2.3} and \eqref{eq:2.4} with initial data $n_0$. If $n_{\neq}$ satisfy the Assumption \ref{ass:2.8}, then there exist $\nu_0=\nu(n_0,\alpha)$ and a positive constant $B_3=B(\|\Lambda_y^{\alpha-1} n_0\|_{L^2}, \|n_0\|_{L^2}, B_1, B_2, C_{\neq}, \alpha)$, if $\nu<\nu_0$, we have
$$
\left\|\Lambda_y^{\alpha-1}n^{0}\right\|_{L^\infty(0, T^\ast;L^2)}\leq B_3.
$$
\end{lemma}

\begin{proof}
Let $\beta=\alpha-1$ and applying $\Lambda^\beta_y$ to \eqref{eq:2.3}, one gets
\begin{equation}\label{eq:4.14}
\partial_t\Lambda^\beta_yn^{0}+\nu(-\partial_{yy})^{\alpha/2}\Lambda^\beta_yn^0
+\nu\partial_y\Lambda^\beta_y(n^{0}\mathbf{B}_1(n^0))+\nu\Lambda^\beta_y\left(\nabla\cdot(n_{\neq} \mathbf{B}_2(n_{\neq}))\right)^{0}=0.
\end{equation}
Multiplying both sides of $\eqref{eq:4.14}$ by $\Lambda^\beta_yn^{0}$ and integrate over $\mathbb{T}$, to obtain
\begin{equation}\label{eq:4.1602}
\begin{aligned}
\frac{1}{2}\frac{d}{dt}\left\|\Lambda^\beta_y n^0\right\|^2_{L^2}+\nu\left\|\Lambda^{\beta+\alpha/2}_y n^0\right\|^2_{L^2}
&+\nu\int_{\mathbb{T}}\partial_y\Lambda^\beta_y(n^{0}\mathbf{B}_1(n^0))\Lambda^\beta_y n^0dy\\
&+\nu\int_{\mathbb{T}}\Lambda^\beta_y\left(\nabla\cdot(n_{\neq} \mathbf{B}_2(n_{\neq}))\right)^{0}\Lambda^\beta_y n^0dy=0.
\end{aligned}
\end{equation}
We deduce by Lemma \ref{lem:2.4} and \ref{lem:2.5} that
$$
\begin{aligned}
\left|\int_{\mathbb{T}}\partial_y\Lambda^\beta_y(n^{0}\mathbf{B}_1(n^0))\Lambda^\beta_y n^0dy\right|
&\lesssim\left\|\partial_y\Lambda^{\beta-\alpha/2}_y(n^{0}\mathbf{B}_1(n^0))\right\|_{L^2}\left\|\Lambda^{\beta+\alpha/2}_y n^0\right\|_{L^2}\\
&\lesssim\left\|\Lambda^{\alpha/2}_y(n^{0}\mathbf{B}_1(n^0))\right\|_{L^2}\left\|\Lambda^{\beta+\alpha/2}_y n^0\right\|_{L^2},
\end{aligned}
$$
and 
$$
\left\|\Lambda^{\alpha/2}_y(n^{0}\mathbf{B}_1(n^0))\right\|_{L^2}
\lesssim\left\|\Lambda^{\alpha/2}_yn^{0}\right\|_{L^2}\left\|\mathbf{B}_1(n^0)\right\|_{L^\infty}
+\left\|n^{0}\right\|_{L^\infty}\left\|\Lambda^{\alpha/2}_y\mathbf{B}_1(n^0)\right\|_{L^2}.
$$
Combining energy estimate and the definition of $\mathbf{B}_1(n^0)$ in \eqref{eq:2.5}, one has
$$
\begin{aligned}
\left\|\Lambda^{\alpha/2}_y\mathbf{B}_1(n^0)\right\|_{L^2}
\lesssim\left\|\mathbf{B}_1(n^0)\right\|^{1/2}_{L^2}\left\|\Lambda^{\alpha}_y\mathbf{B}_1(n^0)\right\|^{1/2}_{L^2}
\lesssim\left\|\mathbf{B}_1(n^0)\right\|^{1/2}_{L^\infty}\left\|\Lambda^{\beta}_yn^0\right\|^{1/2}_{L^2},
\end{aligned}
$$
and
$$
\left\|\Lambda^{\alpha/2}_yn^{0}\right\|_{L^2}
\lesssim\left\|\Lambda^{\beta}_yn^{0}\right\|^{2\beta/\alpha}_{L^2}
\left\|\Lambda^{\beta+\alpha/2}_yn^{0}\right\|^{1-2\beta/\alpha}_{L^2}.
$$
Thus, we deduce by Lemma \ref{lem:4.1}, \ref{lem:4.3} and \ref{lem:A.1} that
$$
\begin{aligned}
&\left\|\Lambda^{\alpha/2}_yn^{0}\right\|_{L^2}\big\|\mathbf{B}_1(n^0)\big\|_{L^\infty}\left\|\Lambda^{\beta+\alpha/2}_y n^0\right\|_{L^2}\\
\lesssim&\left\|\Lambda^{\beta}_yn^{0}\right\|^{2\beta/\alpha}_{L^2}
\left\|\Lambda^{\beta+\alpha/2}_yn^{0}\right\|^{2-2\beta/\alpha}_{L^2}\big\|n^0\big\|_{L^2}\\
\leq& \frac{1}{8}\left\|\Lambda^{\beta+\alpha/2}_yn^{0}\right\|^{2}_{L^2}
+K_2\left\|\Lambda^{\beta}_yn^{0}\right\|^{2}_{L^2},
\end{aligned}
$$
and
$$
\begin{aligned}
&\left\|n^{0}\right\|_{L^\infty}\left\|\Lambda^{\alpha/2}_y\mathbf{B}_1(n^0)\right\|_{L^2}\left\|\Lambda^{\beta+\alpha/2}_y n^0\right\|_{L^2}\\
\lesssim&\left\|n^{0}\right\|_{L^\infty}\left\|n^0\right\|^{1/2}_{L^2}
\left\|\Lambda^{\beta}_yn^0\right\|^{1/2}_{L^2}\left\|\Lambda^{\beta+\alpha/2}_y n^0\right\|_{L^2}\\
 \leq& \frac{1}{8}\left\|\Lambda^{\beta+\alpha/2}_yn^{0}\right\|^{2}_{L^2}
+K_3\left\|\Lambda^{\beta}_yn^{0}\right\|_{L^2},
\end{aligned}
$$
where $K_2=K(B_1,\alpha), K_3=K(B_1,B_2)$. Then the third term of \eqref{eq:4.1602} is estimated as
\begin{equation}\label{eq:4.16}
\left|\nu\int_{\mathbb{T}}\partial_y\Lambda^\beta_y(n^{0}\mathbf{B}_1(n^0))\Lambda^\beta_y n^0dy\right|
\leq\frac{\nu}{4}\left\|\Lambda^{\beta+\alpha/2}_yn^{0}\right\|^{2}_{L^2}+\nu K_2\left\|\Lambda^{\beta}_yn^{0}\right\|^{2}_{L^2}
+\nu K_3\left\|\Lambda^{\beta}_yn^{0}\right\|_{L^2}.
\end{equation}
By Lemma \ref{lem:2.2}, \ref{lem:2.4} and \ref{lem:2.5}, one gets
$$
\begin{aligned}
&\left|\int_{\mathbb{T}}\Lambda^\beta_y\left(\nabla\cdot(n_{\neq} \mathbf{B}_2(n_{\neq}))\right)^{0}\Lambda^\beta_y n^0dy\right|\\
=&\left|\frac{1}{2\pi}\int_{\mathbb{T}^2}\partial_y(n_{\neq}\partial_y\nabla^{-1}\cdot \mathbf{B}_2(n_{\neq}))\Lambda^{2\beta}_y n^0dxdy\right|\\
\lesssim&\left\|\Lambda^{\alpha/2-1}\partial_y(n_{\neq}\partial_y\nabla^{-1}\cdot \mathbf{B}_2(n_{\neq}))\right\|_{L^2}\left\|\Lambda^{2\beta+1-\alpha/2}_y n^0\right\|_{L^2}\\
\lesssim&\left\|\Lambda^{\alpha/2}(n_{\neq}\partial_y\nabla^{-1}\cdot \mathbf{B}_2(n_{\neq}))\right\|_{L^2}
\left\|\Lambda^{\beta+\alpha/2}_y n^0\right\|_{L^2},
\end{aligned}
$$
and 
$$
\begin{aligned}
\left\|\Lambda^{\alpha/2}(n_{\neq}\partial_y\nabla^{-1}\cdot \mathbf{B}_2(n_{\neq}))\right\|_{L^2}
\lesssim&\left\|\Lambda^{\alpha/2}n_{\neq}\right\|_{L^2}\big\|\partial_y\nabla^{-1}\cdot \mathbf{B}_2(n_{\neq})\big\|_{L^\infty}\\
&+\big\|n_{\neq}\big\|_{L^\infty}\left\|\Lambda^{\alpha/2}\partial_y\nabla^{-1}\cdot \mathbf{B}_2(n_{\neq})\right\|_{L^2}.
\end{aligned}
$$
Combining
$$
\big\|\partial_y\nabla^{-1}\cdot \mathbf{B}_2(n_{\neq})\big\|_{L^\infty}\lesssim\big\|n_{\neq}\big\|_{L^4},\ \ \
\big\|n_{\neq}\big\|_{L^4}\lesssim  \big\| n_{\neq}\big\|_{L^\infty},
$$
and Lemma \ref{lem:A.1}, one has
$$
\begin{aligned}
\left\|\Lambda^{\alpha/2}\partial_y\nabla^{-1}\cdot \mathbf{B}_2(n_{\neq})\right\|_{L^2}
&\lesssim\big\|\partial_y\nabla^{-1}\cdot \mathbf{B}_2(n_{\neq})\big\|^{1-\alpha/2}_{L^2}\big\|\partial_y \mathbf{B}_2(n_{\neq})\big\|^{\alpha/2}_{L^2}\\
&\lesssim\left\|\partial_y\nabla^{-1}\cdot \mathbf{B}_2(n_{\neq})\right\|^{1-\alpha/2}_{L^\infty}\big\|\partial_y \mathbf{B}_2(n_{\neq})\big\|^{\alpha/2}_{L^2}\\
&\lesssim\big\|n_{\neq}\big\|^{1-\alpha/2}_{L^4}\big\|n_{\neq}\big\|^{\alpha/2}_{L^2}.
\end{aligned}
$$
Then we can easily get 
$$
\begin{aligned}
&\left\|\Lambda^{\alpha/2}(n_{\neq}\partial_y\nabla^{-1}\cdot \mathbf{B}_2(n_{\neq}))\right\|_{L^2}
\left\|\Lambda^{\beta+\alpha/2}_y n^0\right\|_{L^2}\\
\lesssim&\left\|\Lambda^{\alpha/2}n_{\neq}\right\|_{L^2}\big\|n_{\neq}\big\|_{L^4}\left\|\Lambda^{\beta+\alpha/2}_y n^0\right\|_{L^2}
+\big\|n_{\neq}\big\|_{L^\infty}\big\|n_{\neq}\big\|^{1-\alpha/2}_{L^4}\left\|n_{\neq}\right\|^{\alpha/2}_{L^2}
\left\|\Lambda^{\beta+\alpha/2}_y n^0\right\|_{L^2}\\
\leq &\frac{1}{4}\left\|\Lambda^{\beta+\alpha/2}_y n^0\right\|_{L^2}+ K_4\left(\left\|\Lambda^{\alpha/2}n_{\neq}\right\|^2_{L^2}+\big\|n_{\neq}\big\|^\alpha_{L^2}\right),
\end{aligned}
$$
where $K_4=K(B_2)$. Therefore, the fourth term of \eqref{eq:4.1602} is estimated as
\begin{equation}\label{eq:4.18}
\left|\nu\int_{\mathbb{T}}\Lambda^\beta_y\left(\nabla\cdot(n_{\neq} \mathbf{B}_2(n_{\neq}))\right)^{0}\Lambda^\beta_y n^0dy\right|\leq\frac{\nu}{4}\left\|\Lambda^{\beta+\alpha/2}_y n^0\right\|_{L^2}+ \nu K_4\left(\left\|\Lambda^{\alpha/2}n_{\neq}\right\|^2_{L^2}+\big\|n_{\neq}\big\|^\alpha_{L^2}\right).
\end{equation}
Combining \eqref{eq:4.1602}, \eqref{eq:4.16} and \eqref{eq:4.18}, one has
\begin{equation}\label{eq:4.19}
\begin{aligned}
\frac{d}{dt}\left\|\Lambda^\beta_y n^0\right\|^2_{L^2}+\nu\left\|\Lambda^{\beta+\alpha/2}_y n^0\right\|^2_{L^2}
&\leq \nu K_5\left\|\Lambda^{\beta}_yn^{0}\right\|^{2}_{L^2}
+\nu K_5\left\|\Lambda^{\beta}_yn^{0}\right\|_{L^2}\\
&\ \ \ +\nu K_5\left(\left\|\Lambda^{\alpha/2}n_{\neq}\right\|^2_{L^2}+\big\|n_{\neq}\big\|^\alpha_{L^2}\right),
\end{aligned}
\end{equation}
where $K_5=K_2+K_3+K_4$. Since $\|n^0\|_{L^2}\leq B_1$, we imply 
$$
-\left\|\Lambda^{\beta+\alpha/2}_{y}n^{0}\right\|^{2}_{L^2}\leq -\frac{\big\|\Lambda^\beta_yn^{0}\big\|^{\frac{3\alpha-2}{\alpha-1}}_{L^2}}{C\big\|n^0\big\|_{L^2}^{\frac{\alpha}{\alpha-1}}}
\leq -\frac{\big\|\Lambda^\beta_yn^{0}\big\|^{\frac{3\alpha-2}{\alpha-1}}_{L^2}}{CB^{\frac{\alpha}{\alpha-1}}_1}.
$$
Then the \eqref{eq:4.19} can be written as
\begin{equation}\label{eq:4.20}
\begin{aligned}
\frac{d}{dt}\left\|\Lambda^\beta_y n^0\right\|^2_{L^2}
&\leq-\frac{\big\|\Lambda^\beta_yn^{0}\big\|^{2}_{L^2}}{CB^{\frac{\alpha}{\alpha-1}}_1}
\left(\left\|\Lambda^\beta_y n^0\right\|^{1+\frac{1}{\alpha-1}}_{L^2}-CK_5B_1^{\frac{\alpha}{\alpha-1}}-CK_5B_1^{\frac{\alpha}{\alpha-1}}\left\|\Lambda^\beta_y n^0\right\|^{-1}_{L^2}\right)\\
&\ \ \ \ \ +\nu K_5\left(\left\|\Lambda^{\alpha/2}n_{\neq}\right\|^2_{L^2}+\big\|n_{\neq}\big\|^\alpha_{L^2}\right).
\end{aligned}
\end{equation}
Define
\begin{equation}\label{eq:4.21}
G(t)=\nu K_5\int_{0}^t\left\|\Lambda^{\alpha/2}n_{\neq}\right\|^2_{L^2}+\big\|n_{\neq}\big\|^\alpha_{L^2}d\tau,
\end{equation}
and by Assumption \ref{ass:2.8}, \eqref{eq:A.1601}, \eqref{eq:A.1602} and $\nu$ is small, one has
$$
\begin{aligned}
G(t)&\leq \nu K_5 \int_{0}^{s_0}\left\|\Lambda^{\alpha/2}n_{\neq}\right\|^2_{L^2}+\big\|n_{\neq}\big\|^\alpha_{L^2}d\tau+\nu
K_5\int_{s_0}^t\left\|\Lambda^{\alpha/2}n_{\neq}\right\|^2_{L^2}+\big\|n_{\neq}\big\|^\alpha_{L^2}d\tau\\
&\lesssim K_5\left(\big\|n_0\big\|^{\frac{\alpha}{\alpha-1}}_{L^2}+M\right)\big\|n_0\big\|^2_{L^2}
+K_5\big\|n_0\big\|^{\alpha}_{L^2}\\
&\ \ \ \ +16C_{\neq}K_4\big\|n_0\big\|^2_{L^2}+2\nu\lambda^{-1}_{\nu,\alpha}\alpha^{-1}K_5(4C_{\neq})^{\alpha/2}\big\|n_0\big\|^\alpha_{L^2}\\
&\leq K_5\left(\big\|n_0\big\|^{\frac{\alpha}{\alpha-1}}_{L^2}+M+32C_{\neq}\right)\big\|n_0\big\|^2_{L^2}
+K_5\big\|n_0\big\|^{\alpha}_{L^2}\\
&\leq K_6\big\|n_0\big\|^2_{L^2}+K_5\big\|n_0\big\|^{\alpha}_{L^2},
\end{aligned}
$$
where $K_6=K(\|n_0\|_{L^2}, \alpha,C_{\neq}, M, K_5)$. Combining \eqref{eq:4.20} and \eqref{eq:4.21}, we obtain
\begin{equation}\label{eq:4.23}
\begin{aligned}
&\frac{d}{dt}\left(\left\|\Lambda^\beta_y n^0\right\|^2_{L^2}-G(t)\right)\\
\leq&-\frac{\big\|\Lambda^\beta_yn^{0}\big\|^{2}_{L^2}}{CB^{\frac{\alpha}{\alpha-1}}_1}
\left(\left\|\Lambda^\beta_y n^0\right\|^{1+\frac{1}{\alpha-1}}_{L^2}-CK_5B_1^{\frac{\alpha}{\alpha-1}}-CK_5B_1^{\frac{\alpha}{\alpha-1}}
\left\|\Lambda^\beta_y n^0\right\|^{-1}_{L^2}\right)\\
\leq&-\frac{\big\|\Lambda^\beta_yn^{0}\big\|^{2}_{L^2}}{CB^{\frac{\alpha}{\alpha-1}}_1}
\left(\left\|\Lambda^\beta_y n^0\right\|^{1+\frac{1}{\alpha-1}}_{L^2}-G(t)-CK_5B_1^{\frac{\alpha}{\alpha-1}}-CK_5B_1^{\frac{\alpha}{\alpha-1}}
\left\|\Lambda^\beta_y n^0\right\|^{-1}_{L^2}\right).
\end{aligned}
\end{equation}
If
$$
\left\|\Lambda^\beta_y n^0\right\|^{1+\frac{1}{\alpha-1}}_{L^2}-G(t)-CK_5B_1^{\frac{\alpha}{\alpha-1}}-CK_5B_1^{\frac{\alpha}{\alpha-1}}
\left\|\Lambda^\beta_y n^0\right\|^{-1}_{L^2}\geq0,
$$
then we deduce by \eqref{eq:4.23} that
\begin{equation}\label{eq:4.24}
\left\|\Lambda^\beta_y n^0\right\|^2_{L^2}\leq\left\|\Lambda^\beta_y n_0\right\|^2_{L^2}+K_6\big\|n_0\big\|^2_{L^2}+K_5\big\|n_0\big\|^{\alpha}_{L^2}.
\end{equation}
If
\begin{equation}\label{eq:4.25}
\left\|\Lambda^\beta_y n^0\right\|^{1+\frac{1}{\alpha-1}}_{L^2}-G(t)-CK_5B_1^{\frac{\alpha}{\alpha-1}}-CK_5B_1^{\frac{\alpha}{\alpha-1}}
\left\|\Lambda^\beta_y n^0\right\|^{-1}_{L^2}<0,
\end{equation}
we imply that there exist a positive constant
$$
K_7=K(\alpha,B_1, K_5, C_{\neq}, \|n_0\|_{L^2},M)<\infty,
$$
such that
\begin{equation}\label{eq:4.26}
\left\|\Lambda^\beta_y n^0\right\|^2_{L^2}\leq K_7.
\end{equation}
Combining \eqref{eq:4.24} and \eqref{eq:4.26}, we have
$$
\left\|\Lambda^\beta_y n^0\right\|^2_{L^2}\leq\left\|\Lambda^\beta_y n_0\right\|^2_{L^2}+CK_6\big\|n_0\big\|^2_{L^2}+CK_5\big\|n_0\big\|^{\alpha}_{L^2}
+K_7\triangleq B_3^2.
$$
This completes the proof of Lemma \ref{lem:4.4}.
\end{proof}

Finally, we prove the Proposition \ref{prop:2.9} by the enhanced dissipation of shear flow.

\begin{proof}[The proof of Proposition \ref{prop:2.9}]
We will prove (P-1) and (P-2) in the Proposition \ref{prop:2.9} respectively.

\vskip .05in

\noindent (1) {\em Nonzero mode $L^2\dot{H}^{\alpha/2}$ estimate of $n_{\neq}$}.  Let us multiply both sides of \eqref{eq:2.4} by $n_{\neq}$ and integrate over $\mathbb{T}^2$, we obtain
\begin{equation}\label{eq:4.27}
\frac{1}{2}\frac{d}{dt}\big\|n_{\neq}\big\|^2_{L^2}+\nu\left\|\Lambda^{\alpha/2}n_{\neq}\right\|^2_{L^2}
+\nu\int_{\mathbb{T}^2}F(n^0,n_{\neq})n_{\neq}dxdy=0,
\end{equation}
where
\begin{equation}\label{eq:4.2901}
\begin{aligned}
F(n^0,n_{\neq})&=\nabla n^{0}\cdot\mathbf{B}_2(n_{\neq})+\partial_yn_{\neq}\mathbf{B}_1(n^0)\\
&\ \ \ +\left(\nabla\cdot(n_{\neq}\mathbf{B}_2(n_{\neq}))\right)_{\neq}-n^{0}n_{\neq}-n_{\neq}(n^0-\overline{n}),
\end{aligned}
\end{equation}
and $\mathbf{B}_1(n^0)$ and $\mathbf{B}_2(n_{\neq})$ are defined in \eqref{eq:2.5}. By Lemma \ref{lem:2.5}, we obtain
$$
\begin{aligned}
\left|\int_{\mathbb{T}^2}\nabla n^{0}\cdot \mathbf{B}_2(n_{\neq})n_{\neq}dxdy\right|
&=\left|\int_{\mathbb{T}^2}\partial_y n^{0}\partial_y\nabla^{-1}\cdot\mathbf{B}_2(n_{\neq})n_{\neq}dxdy\right|\\
&\lesssim \left\|\Lambda_y^{\beta}n^{0}\right\|_{L^2}\left\|\Lambda^{1-\beta}
\left(\partial_y\nabla^{-1}\cdot\mathbf{B}_2(n_{\neq})n_{\neq}\right)  \right\|_{L^2}.
\end{aligned}
$$
Combining energy estimate Lemma \ref{lem:2.4} and \ref{lem:A.1}, one has
$$
\begin{aligned}
\left\|\Lambda^{1-\beta}
\left(\partial_y\nabla^{-1}\cdot\mathbf{B}_2(n_{\neq})n_{\neq}\right)\right\|_{L^2}
&\lesssim \big\|n_{\neq}\big\|_{L^4}\left\|\Lambda^{1-\beta}
\partial_y\nabla^{-1}\cdot\mathbf{B}_2(n_{\neq})\right\|_{L^4}\\
&\ \ \ +\left\|\partial_y\nabla^{-1}\cdot\mathbf{B}_2(n_{\neq})\right\|_{L^\infty}
\left\|\Lambda^{1-\beta}n_{\neq}\right\|_{L^2}\\
&\lesssim \big\|n_{\neq}\big\|^2_{L^4}+\big\|n_{\neq}\big\|_{L^4}\left\|\Lambda^{1-\beta}n_{\neq}\right\|_{L^2},
\end{aligned}
$$
and
$$
\begin{aligned}
&\left\|\Lambda_y^{\beta}n^{0}\right\|_{L^2}\left\|\Lambda^{1-\beta}
\left(\partial_y\nabla^{-1}\cdot\mathbf{B}_2(n_{\neq})n_{\neq}\right)\right\|_{L^2}\\
\lesssim& \left\|\Lambda_y^{\beta}n^{0}\right\|_{L^2}\big\|n_{\neq}\big\|^2_{L^4}
+\left\|\Lambda_y^{\beta}n^{0}\right\|_{L^2}\big\|n_{\neq}\big\|_{L^4}\left\|\Lambda^{1-\beta}n_{\neq}\right\|_{L^2}\\
\lesssim&\left\|\Lambda_y^{\beta}n^{0}\right\|_{L^2}\big\|n_{\neq}\big\|^{2-2/\alpha}_{L^2}
\left\|\Lambda^{\alpha/2}n_{\neq}\right\|^{2/\alpha}_{L^2}
+\left\|\Lambda_y^{\beta}n^{0}\right\|_{L^2}\big\|n_{\neq}\big\|^{4-5/\alpha}_{L^2}
\left\|\Lambda^{\alpha/2}n_{\neq}\right\|^{5/\alpha-2}_{L^2}\\
\leq&\frac{1}{2}\left\|\Lambda^{\alpha/2}n_{\neq}\right\|^{2}_{L^2}+C\left(\left\|\Lambda^\beta n^{0}\right\|^{\frac{\alpha}{\alpha-1}}_{L^2}+\left\|\Lambda^\beta n^{0}\right\|^{\frac{2\alpha}{4\alpha-5}}_{L^2}\right)\big\|n_{\neq}\big\|^{2}_{L^2}
\end{aligned}
$$
Thus, we have
\begin{equation}\label{eq:4.28}
\left|\nu\int_{\mathbb{T}^2}\nabla n^{0}\cdot \mathbf{B}_2(n_{\neq})n_{\neq}dxdy\right|
\leq\frac{\nu}{2}\left\|\Lambda^{\alpha/2}n_{\neq}\right\|^{2}_{L^2}+C\nu\left(\left\|\Lambda^\beta n^{0}\right\|^{\frac{\alpha}{\alpha-1}}_{L^2}+\left\|\Lambda^\beta n^{0}\right\|^{\frac{2\alpha}{4\alpha-5}}_{L^2}\right)\big\|n_{\neq}\big\|^{2}_{L^2}.
\end{equation}
By Lemma \ref{lem:A.1}, one gets
\begin{equation}\label{eq:4.29}
\left|\nu\int_{\mathbb{T}^2}\partial_yn_{\neq}\mathbf{B}_1(n^0)n_{\neq}dxdy\right|
=\left|\frac{\nu}{2}\int_{\mathbb{T}^2}n^2_{\neq}\partial_y\mathbf{B}_1(n^0)dxdy\right|
\lesssim\nu \|n_{\neq}\|^2_{L^2}\|n^0\|_{L^\infty}.
\end{equation}
Combining the definition of $P_{\neq}$ in \eqref{eq:1.9} and Lemma \ref{lem:A.1}, one has
\begin{equation}\label{eq:4.30}
\begin{aligned}
\left|\nu\int_{\mathbb{T}^2}\left(\nabla\cdot(n_{\neq} \mathbf{B}_2(n_{\neq}))\right)_{\neq}n_{\neq}dxdy\right|
&=\left|\nu\int_{\mathbb{T}^2}\nabla\cdot(n_{\neq} \mathbf{B}_2(n_{\neq}))n_{\neq}dxdy\right|\\
&=\left|\frac{\nu}{2}\int_{\mathbb{T}^2}\nabla\cdot\mathbf{B}_2(n_{\neq})n^2_{\neq}dxdy\right|\\
&\lesssim \nu \big\|n_{\neq}\big\|^2_{L^2}\big\|n^0\big\|_{L^\infty}.
\end{aligned}
\end{equation}
And we can easily get
\begin{equation}\label{eq:4.31}
\left|\nu\int_{\mathbb{T}^2}-n^{0}(n_{\neq})^2-(n_{\neq})^2(n^0-\overline{n})dxdy\right|
\lesssim \nu\big\|n^0\big\|_{\infty}\big\| n_{\ne}\big\|_{L^2}^2.
\end{equation}
Thus, we obtain from \eqref{eq:4.28}-\eqref{eq:4.31} that
\begin{equation}\label{eq:4.32}
\begin{aligned}
&\left|\nu\int_{\mathbb{T}^2}F(n^0,n_{\neq})n_{\neq}dxdy\right|\\
\leq&\frac{\nu}{2}\left\|\Lambda^{\alpha/2}n_{\neq}\right\|^{2}_{L^2}+C\nu\left(\left\|\Lambda^\beta n^{0}\right\|^{\frac{\alpha}{\alpha-1}}_{L^2}+\left\|\Lambda^\beta n^{0}\right\|^{\frac{2\alpha}{4\alpha-5}}_{L^2}+\| n^0\|_{\infty}\right)\big\|n_{\neq}\big\|^{2}_{L^2}.
\end{aligned}
\end{equation}
Combining \eqref{eq:4.27}, \eqref{eq:4.32}, Lemma \ref{lem:4.3} and \ref{lem:4.4}, one has
\begin{equation}\label{eq:4.33}
\frac{d}{dt}\big\|n_{\neq}\big\|^2_{L^2}+\nu\left\|\Lambda^{\alpha/2}n_{\neq}\right\|^2_{L^2}
\leq \nu K_8\big\|n_{\neq}\big\|^{2}_{L^2},
\end{equation}
where
$$
K_8=K(\alpha, B_2,B_3)<\infty.
$$
Then by time integral in \eqref{eq:4.33} from $s$ to $t$, the $\nu$ is small and Assumption \ref{ass:2.8}, one gets
$$
\begin{aligned}
&\big\|n_{\neq}(t)\big\|^2_{L^2}+\nu\int_{s}^{t}\left\|\Lambda^{\alpha/2} n_{\neq}(\tau)\right\|^2_{L^2}d\tau\\
\leq& 4C_{\neq} \nu K_8  \big\|n_0\big\|^2_{L^2}\int_{s}^{t}e^{-\lambda_{\nu,\alpha} (\tau-s_0)}d\tau+\big\|n_{\neq}(s)\big\|^2_{L^2}\\
\leq& 4C_{\neq}C \nu K_8 \lambda^{-1}_{\nu,\alpha} e^{-\lambda_{\nu,\alpha} (s-s_0)}\big\|n_0\big\|^2_{L^2} +4C_{\neq}e^{-\lambda_{\nu,\alpha} (s-s_0)}\big\|n_0\big\|^2_{L^2}\\
\leq& 8C_{\neq}e^{-\lambda_{\nu,\alpha} (s-s_0)}\big\|n_0\big\|^2_{L^2}.
\end{aligned}
$$
Thus, we have
\begin{equation}\label{eq:4.3501}
\nu\int_{s}^{t}\left\|\Lambda^{\alpha/2} n_{\neq}(\tau)\right\|^2_{L^2}d\tau\leq 8C_{\neq}e^{-\lambda_{\nu,\alpha} (s-s_0)}\big\|n_0\big\|^2_{L^2}.
\end{equation}

\vskip .05in

\noindent (2) {\em Nonzero mode $L^\infty L^{2}$ estimate of $n_{\neq}$}. Denote
$$
\mathcal{S}_{t}=e^{-tL_{\nu,\alpha}},
$$
where the $L_{\nu,\alpha}$ is defined in \eqref{eq:1.1101}. Combining Duhamel's principle, \eqref{eq:2.4} and \eqref{eq:4.2901}, one has
$$
n_{\neq}(t+s)=\mathcal{S}_{t}n_{\neq}(s)
-\nu\int_{s}^{t+s}\mathcal{S}_{t+s-\tau}F(n^0,n_{\neq})d\tau.
$$
Then we have
\begin{equation}\label{eq:4.36}
\begin{aligned}
\big\|n(t+s)\big\|_{L^2}\leq&\big\|\mathcal{S}_{t}n_{\neq}(s)\big\|_{L^2}+C\nu\int_{s}^{s+t}
\left\|\nabla n^{0}\cdot\mathbf{B}_2(n_{\neq})\right\|_{L^2}
+\left\|\left(\nabla\cdot(n_{\neq}\mathbf{B}_2(n_{\neq}))\right)_{\neq}\right\|_{L^2}\\
&+\big\|n^{0}n_{\neq}\big\|_{L^2}+\big\|n_{\neq}(n^0-\overline{n})\big\|_{L^2}d\tau
+C\nu\int_{s}^{t+s}\left\|\mathcal{S}_{t+s-\tau}\left(\partial_yn_{\neq}\mathbf{B}_1(n^0)\right)\right\|_{L^2}d\tau.\!\!\!
\end{aligned}
\end{equation}
By Corollary \ref{cor:1.2}, one gets
\begin{equation}\label{eq:4.37}
\big\|\mathcal{S}_{t}n_{\neq}(s)\big\|_{L^2}\leq e^{-\lambda_{\nu,\alpha} t+\pi/2}\big\|n_{\neq}(s)\big\|_{L^2}.
\end{equation}
Combining Lemma \ref{lem:A.1} and energy estimate, one has
$$
\big\|\nabla n^{0}\cdot\mathbf{B}_2(n_{\neq})\big\|_{L^2}
\lesssim\big\|\partial_y n^0\big\|_{L^2}\big\|n_{\neq}\big\|_{L^4}
\lesssim \big\|\partial_y n^0\big\|_{L^2}\big\|n_{\neq}\big\|_{L^2}+\big\|\partial_y n^0\big\|_{L^2}\left\|\Lambda^{\alpha/2}n_{\neq}\right\|_{L^2},
$$
then we deduce by Assumption \ref{ass:3.1} and Lemma \ref{eq:A.4} that
\begin{equation}\label{eq:4.38}
\begin{aligned}
&\nu\int_{s}^{s+t}\big\|\nabla n^{0}\cdot\mathbf{B}_2(n_{\neq})\big\|_{L^2}d\tau\\
\lesssim&\nu\int_{s}^{s+t}
\big\|\partial_y n^0\big\|_{L^2}\big\|n_{\neq}\big\|_{L^2}d\tau+\nu\int_{s}^{s+t}
\big\|\partial_y n^0\big\|_{L^2}\left\|\Lambda^{\alpha/2}n_{\neq}\right\|_{L^2}d\tau\\
\lesssim&\left(\nu\int_{s}^{t+s}\big\|\partial_y n^0\big\|^2_{L^2}d\tau\right)^{1/2}
 \left(\nu\int_{s}^{t+s}\big\|n_{\neq}\big\|^2_{L^2}d\tau\right)^{1/2}\\
 &+\left(\nu\int_{s}^{t+s}\big\|\partial_y n^0\big\|^2_{L^2}d\tau\right)^{1/2}
 \left(\nu\int_{s}^{t+s}\left\|\Lambda^{{\alpha/2}} n_{\neq}\right\|^2_{L^2}d\tau\right)^{1/2}\\
 \lesssim&\left(\nu^{6/7}t\right)^{1/2}\left(\left(\lambda^{-1}_{\nu,\alpha} \nu\right)^{1/2}+4\right)C^{1/2}_{\neq}e^{-\lambda_{\nu,\alpha}(s-s_0)/2}\big\|n_{0}\big\|_{L^2}.
\end{aligned}
\end{equation}
By Lemma \ref{lem:A.1} and energy estimate, one has
$$
\begin{aligned}
\left\|\left(\nabla\cdot(n_{\neq}\mathbf{B}_2(n_{\neq}))\right)_{\neq}\right\|_{L^2}
&\lesssim \big\|\nabla n_{\neq}\big\|_{L^2}\big\|n_{\neq}\big\|_{L^4}+\big\|n_{\neq}\big\|_{L^\infty}\big\|n_{\neq}\big\|_{L^2}\\
&\lesssim  \big\|\nabla n_{\neq}\big\|_{L^2} \big\|n_{\neq}\big\|_{L^2}+\big\|\nabla n_{\neq}\big\|_{L^2}
\left\|\Lambda^{\alpha/2}n_{\neq}\right\|_{L^2}+\big\|n_{\neq}\big\|_{L^\infty}\big\|n_{\neq}\big\|_{L^2},
\end{aligned}
$$
and we deduce by Assumption \ref{ass:3.1}, Lemma \ref{lem:4.3} and \ref{lem:A.4} that
$$
\begin{aligned}
\nu\int_{s}^{t+s} \big\|\nabla n_{\neq}\big\|_{L^2} \big\|n_{\neq}\big\|_{L^2}d\tau
&\lesssim\left(\nu\int_{s}^{t+s}\big\|\nabla n_{\neq}\big\|^2_{L^2}d\tau\right)^{1/2}
\left(\nu\int_{s}^{t+s}\big\|n_{\neq}\big\|^2_{L^2}d\tau\right)^{1/2}\\
&\lesssim \left(\nu^{6/7}t\right)^{1/2}\left(\lambda^{-1}_{\nu,\alpha} \nu\right)^{1/2}C^{1/2}_{\neq}e^{-\lambda_{\nu,\alpha}(s-s_0)/2}\big\|n_0\big\|_{L^2},
\end{aligned}
$$
$$
\begin{aligned}
\nu\int_{s}^{t+s} \big\|\nabla n_{\neq}\big\|_{L^2}
\left\|\Lambda^{\alpha/2}n_{\neq}\right\|_{L^2}d\tau
&\lesssim\left(\nu\int_{s}^{t+s}\big\|\nabla n_{\neq}\big\|^2_{L^2}d\tau\right)^{1/2}
\left(\nu\int_{s}^{t+s}\left\|\Lambda^{\alpha/2}n_{\neq}\right\|^2_{L^2}d\tau\right)^{1/2}\\
&\lesssim \left(\nu^{6/7}t\right)^{1/2}C^{1/2}_{\neq}e^{-\lambda_{\nu,\alpha}(s-s_0)/2}\big\|n_{0}\big\|_{L^2},
\end{aligned}
$$
and
$$
\begin{aligned}
\nu\int_{s}^{t+s} \big\|n_{\neq}\big\|_{L^\infty}\big\|n_{\neq}\big\|_{L^2}d\tau
&\lesssim \left(\nu\int_{s}^{t+s}\big\| n_{\neq}\big\|^2_{L^\infty}d\tau\right)^{1/2}
\left(\nu\int_{s}^{t+s}\big\|n_{\neq}\big\|^2_{L^2}d\tau\right)^{1/2}\\
&\lesssim B_2\nu (\lambda^{-1}_{\nu,\alpha} t)^{1/2}C^{1/2}_{\neq}e^{-\lambda_{\nu,\alpha}(s-s_0)/2}\big\|n_0\big\|_{L^2}.
\end{aligned}
$$
Then we have
\begin{equation}\label{eq:4.39}
\begin{aligned}
&\nu\int_{s}^{s+t}\left\|\left(\nabla\cdot(n_{\neq}\mathbf{B}_2(n_{\neq}))\right)_{\neq}\right\|_{L^2}d\tau\\
\lesssim&\left(\left(\nu^{6/7}t\right)^{1/2}\left(\lambda^{-1}_{\nu,\alpha} \nu\right)^{1/2}
+\left(\nu^{6/7}t\right)^{1/2}+B_2\nu (\lambda^{-1}_{\nu,\alpha} t)^{1/2}\right)C^{1/2}_{\neq}e^{-\lambda_{\nu,\alpha}(s-s_0)/2}\big\|n_0\big\|_{L^2}.
\end{aligned}
\end{equation}
Combining Assumption \ref{ass:2.8} and Lemma \ref{lem:4.3}, one has
\begin{equation}\label{eq:4.40}
\begin{aligned}
\nu\int_{s}^{s+t}\big\|n^{0}n_{\neq}\big\|_{L^2}+\big\|n_{\neq}(n^0-\overline{n})\big\|_{L^2}d\tau
&\lesssim\nu\int_{s}^{s+t}\big\|n^{0}\big\|_{L^\infty}\big\|n_{\neq}\big\|_{L^2}d\tau\\
&\lesssim B_2\lambda^{-1}_{\nu,\alpha}\nu C_{\neq}^{1/2}e^{-\lambda_{\nu,\alpha}(s-s_0)/2}\big\|n_0\big\|_{L^2}.
\end{aligned}
\end{equation}
Next, we consider the last term of the right-hand side of \eqref{eq:4.36}. Since
$$
\mathcal{S}_{t+s-\tau}\left(\partial_yn_{\neq}\mathbf{B}_1(n^0)\right)
=\mathcal{S}_{t+s-\tau}\partial_y\left(n_{\neq}\mathbf{B}_1(n^0)\right)
-\mathcal{S}_{t+s-\tau}\left(n_{\neq}\partial_y\mathbf{B}_1(n^0)\right),
$$
one has
$$
\begin{aligned}
&\nu\int_{s}^{t+s}\big\|\mathcal{S}_{t+s-\tau}\left(\partial_yn_{\neq}\mathbf{B}_1(n^0)\right)\big\|_{L^2}d\tau\\
\lesssim&\nu\int_{s}^{t+s}\big\|\mathcal{S}_{t+s-\tau}\partial_y\left(n_{\neq}\mathbf{B}_1(n^0)\right)\big\|_{L^2}d\tau
+\nu\int_{s}^{t+s}\big\|\mathcal{S}_{t+s-\tau}\left(n_{\neq}\partial_y\mathbf{B}_1(n^0)\right)\big\|_{L^2}d\tau.
\end{aligned}
$$
By Lemma \ref{lem:B.4}, we have
$$
\begin{aligned}
\mathcal{S}_{t+s-\tau}\partial_y\left(n_{\neq}\mathbf{B}_1(n^0)\right)(\tau)
=&-\Lambda^{\alpha/2}_{y}\mathcal{S}_{t+s-\tau}\Lambda^{-\alpha/2}_{y}\partial_{y}\left(n_{\neq}\mathbf{B}_1(n^0)\right)(\tau)\\
&-\int_{\tau}^{t+s}e^{-(t+s-\tau_0)\mathcal{L}_{\nu,\alpha}}\mathcal{R}_{n_{\neq}\mathbf{B}_1(n^0)}(\tau_0)d\tau_0,
\end{aligned}
$$
where
\begin{equation}\label{eq:4.4101}
\mathcal{R}_{n_{\neq}\mathbf{B}_1(n^0)}=C_{\alpha}\sum_{k\in\mathbb{Z}}P.V.\int_{\mathbb{T}}\mathcal{K}(y,\widetilde{y})
\Lambda_{x}^{\alpha/2-1}\partial_{x}\mathcal{S}_{\tau_0}\left(\Lambda^{1-\alpha/2}_{x}\Lambda^{-\alpha/2}_{\widetilde{y}}
\partial_{\widetilde{y}}\left(n_{\neq}\mathbf{B}_1(n^0)\right)\right)d\widetilde{y}.
\end{equation}
Then we have
\begin{equation}\label{eq:4.41}
\begin{aligned}
\big\|\mathcal{S}_{t+s-\tau}\partial_y\left(n_{\neq}\mathbf{B}_1(n^0)\right)(\tau)\big\|_{L^2}
\lesssim&\left\|\Lambda^{\alpha/2}_{y}\mathcal{S}_{t+s-\tau}\Lambda^{-\alpha/2}_{y}\partial_{y}
\left(n_{\neq}\mathbf{B}_1(n^0)\right)(\tau)\right\|_{L^2}\\
&+\int_{\tau}^{t+s}\left\|\mathcal{R}_{n_{\neq}\mathbf{B}_1(n^0)}(\tau_0)\right\|_{L^2}d\tau_0.
\end{aligned}
\end{equation}
By Lemma \ref{eq:B.1}, one has
$$
\left\|\Lambda^{\alpha/2}_{y}\mathcal{S}_{t+s-\tau}\Lambda^{-\alpha/2}_{y}\partial_{y}
\left(n_{\neq}\mathbf{B}_1(n^0)\right)(\tau)\right\|_{L^2}
\lesssim\left\|\Lambda^{\alpha/2}_{y}\mathcal{S}_{t+s-\tau}\right\|
\left\|\Lambda^{1-\alpha/2}_{y}\left(n_{\neq}\mathbf{B}_1(n^0)\right)(\tau)\right\|_{L^2},
$$
we deduce by Lemma \ref{lem:A.1} that
$$
\begin{aligned}
\left\|\Lambda^{1-\alpha/2}_{y}\left(n_{\neq}\mathbf{B}_1(n^0)\right)(\tau)\right\|_{L^2}
&\lesssim\left\|\Lambda^{1-\alpha/2}n_{\neq}\right\|_{L^2}
\big\|\mathbf{B}_1(n^0)\big\|_{L^\infty}
+\big\|n_{\neq}\big\|_{L^2}\left\|\Lambda^{1-\alpha/2}_{y}\mathbf{B}_1(n^0)\right\|_{L^\infty}\\
&\lesssim \big\|n_{\neq}\big\|^{2-2/\alpha}_{L^2}\left\|\Lambda^{\alpha/2}n_{\neq}\right\|^{2/\alpha-1}_{L^2}
\big\|n^0\big\|_{L^2}+\big\|n_{\neq}\big\|_{L^2}\big\|n^0\big\|_{L^2}.
\end{aligned}
$$
Since
$$
\begin{aligned}
&\nu\int_{s}^{t+s}\left\|\Lambda^{\alpha/2}_{y}\mathcal{S}_{t+s-\tau}\right\|\big\|n_{\neq}\big\|^{2-2/\alpha}_{L^2}
\left\|\Lambda^{\alpha/2}n_{\neq}\right\|^{2/\alpha-1}_{L^2}
\big\|n^0\big\|_{L^2}d\tau\\
\lesssim& B_1\left(\nu\int_{s}^{t+s}\left\|\Lambda^{\alpha/2}\mathcal{S}_{t+s-\tau}\right\|^2 d\tau\right)^{1/2}
\left(\nu\int_{s}^{t+s} \big\|n_{\neq}\big\|^{4-4/\alpha}_{L^2}\left\|\Lambda^{\alpha/2}n_{\neq}\right\|^{4/\alpha-2}_{L^2}d\tau\right)^{1/2}\\
\lesssim& B_1\left(\nu\int_{s}^{t+s} \big\|n_{\neq}\big\|^{2}_{L^2}d\tau\right)^{1-1/\alpha}
\left(\nu\int_{s}^{t+s} \left\|\Lambda^{\alpha/2}n_{\neq}\right\|^{2}_{L^2}d\tau\right)^{1/\alpha-1/2}\\
\lesssim& (\nu \lambda^{-1}_{\nu,\alpha})^{1-1/\alpha}B_1C^{1/2}_{\neq}e^{-\lambda_{\nu,\alpha}(s-s_0)/2}\big\|n_0\big\|_{L^2},
\end{aligned}
$$
and
$$
\begin{aligned}
&\nu\int_{s}^{t+s}\left\|\Lambda^{\alpha/2}\mathcal{S}_{t+s-\tau}\right\| \big\|n_{\neq}\big\|_{L^2}\big\|n^0\big\|_{L^2}d\tau\\
\lesssim& B_1\left(\nu\int_{s}^{t+s}\left\|\Lambda^{\alpha/2}\mathcal{S}_{t+s-\tau}\right\|^2 d\tau\right)^{1/2}
\left(\nu\int_{s}^{t+s} \big\|n_{\neq}\big\|^{2}_{L^2}d\tau\right)^{1/2}\\
\lesssim &(\nu \lambda^{-1}_{\nu,\alpha})^{1/2} B_1 C^{1/2}_{\neq}e^{-\lambda_{\nu,\alpha}(s-s_0)/2}\big\|n_0\big\|_{L^2},
\end{aligned}
$$
we have
\begin{equation}\label{eq:4.42}
\begin{aligned}
&\nu\int_{s}^{s+t}\left\|\Lambda^{\alpha/2}_{y}\mathcal{S}_{t+s-\tau}\Lambda^{-\alpha/2}_{y}\partial_{y}
\left(n_{\neq}\mathbf{B}_1(n^0)\right)(\tau)\right\|_{L^2}d\tau\\
\lesssim&\left((\nu \lambda^{-1}_{\nu,\alpha})^{1-1/\alpha}+(\nu \lambda^{-1}_{\nu,\alpha})^{1/2}\right)B_1 C^{1/2}_{\neq}e^{-\lambda_{\nu,\alpha}(s-s_0)/2}\big\|n_0\big\|_{L^2}.
\end{aligned}
\end{equation}
By H\"{o}lder's inequality and \eqref{eq:4.4101}, we have
$$
\begin{aligned}
\left\|\mathcal{R}_{n_{\neq}\mathbf{B}_1(n^0)}(\tau_0) \right\|_{L^2}&\lesssim \left(C_{\alpha}\sum_{k\in\mathbb{Z}}P.V.\int_{\mathbb{T}^2}\mathcal{K}^2(y,\widetilde{y})dyd\widetilde{y}\right)^{1/2}\\
&\ \ \cdot\left(\int_{\mathbb{T}^2}\left[\Lambda_{x}^{\alpha/2-1}\partial_{x}\mathcal{S}_{\tau_0}\left(\Lambda^{1-\alpha/2}_{x}\Lambda^{-\alpha/2}_{\widetilde{y}}
\partial_{\widetilde{y}}\left(n_{\neq}\mathbf{B}_1(n^0)\right)\right)\right]^2d\widetilde{y}dx\right)^{1/2}\\
&\lesssim \left\|\Lambda_{x}^{\alpha/2-1}\partial_{x}\mathcal{S}_{\tau_0}\left(\Lambda^{1-\alpha/2}_{x}\Lambda^{-\alpha/2}_{\widetilde{y}}
\partial_{\widetilde{y}}\left(n_{\neq}\mathbf{B}_1(n^0)\right)\right) \right\|_{L^2}\\
&\lesssim \left\|\Lambda^{\alpha/2}\mathcal{S}_{\tau_0}P_{\neq}\right\|\big\|\Lambda^{2-\alpha}\left(n_{\neq}\mathbf{B}_1(n^0)\right) \big\|_{L^2},
\end{aligned}
$$
and combining lemma \ref{lem:2.4} and \ref{lem:A.1}, one has
$$
\begin{aligned}
\big\|\Lambda^{2-\alpha}\left(n_{\neq}\mathbf{B}_1(n^0)\right) \big\|_{L^2}
&\lesssim \big\|\Lambda^{2-\alpha} n_{\neq}\|_{L^2}\|\mathbf{B}_1(n^0)\big\|_{L^\infty}
+\big\|n_{\neq}\big\|_{L^2}\big\|\Lambda^{2-\alpha}\mathbf{B}_1(n^0)\big\|_{L^\infty}\\
&\lesssim \big\|n_{\neq}\big\|^{3-4/\alpha}_{L^2}
\left\|\Lambda^{\alpha/2}n_{\neq}\right\|^{4/\alpha-2}_{L^2}\big\|n^0\big\|_{L^2}
+\big\|n_{\neq}\big\|_{L^2}\big\|n^0\big\|_{L^2}.
\end{aligned}
$$
Since $s\geq s_0$, we have
$$
\begin{aligned}
&\nu\int_{s}^{t+s}\int_{\tau}^{t+s}\left\|\Lambda^{\alpha/2}\mathcal{S}_{\tau_0}P_{\neq}\right\|
\big\|n_{\neq}\big\|^{3-4/\alpha}_{L^2}
\left\|\Lambda^{\alpha/2}n_{\neq}\right\|^{4/\alpha-2}_{L^2}\big\|n^0\big\|_{L^2}d\tau_0d\tau\\
\lesssim&B_1\left[\nu\int_{s}^{t+s}\left(\int_{\tau}^{t+s}\left\|\Lambda^{\alpha/2}\mathcal{S}_{\tau_0}P_{\neq}\right\|d\tau_0
\right)^2d\tau\right]^{1/2}\left[ \nu\int_{s}^{t+s}\big\|n_{\neq}\big\|^{6-8/\alpha}_{L^2}
\left\|\Lambda^{\alpha/2}n_{\neq}\right\|^{8/\alpha-4}_{L^2}d\tau \right]^{1/2}\\
\lesssim&B_1\left(t\lambda_{\nu,\alpha}+1\right)^{1/2}\left(  \nu\int_{s}^{t+s}\big\|n_{\neq}\big\|^{2}_{L^2}
d\tau  \right)^{3/2-2/\alpha} \left(\nu\int_{s}^{t+s}
\left\|\Lambda^{\alpha/2}n_{\neq}\right\|^{2}_{L^2}d\tau\right)^{2/\alpha-1} \\
\lesssim&B_1\left(t\lambda_{\nu,\alpha}+1\right)^{1/2}\left(\nu\lambda^{-1}_{\nu,\alpha}\right)^{3/2-2/\alpha} C^{1/2}_{\neq}e^{-\lambda_{\nu,\alpha}(s-s_0)/2}\big\|n_0\big\|_{L^2}
\end{aligned}
$$
and
$$
\begin{aligned}
&\nu\int_{s}^{t+s}\int_{\tau}^{t+s}\left\|\Lambda^{\alpha/2}\mathcal{S}_{\tau_0}P_{\neq}\right\|
\big\|n_{\neq}\big\|_{L^2}\big\|n^0\big\|_{L^2}d\tau_0d\tau\\
\lesssim& B_1\left[\nu\int_{s}^{t+s}\left(\int_{\tau}^{t+s}\left\|\Lambda^{\alpha/2}\mathcal{S}_{\tau_0}P_{\neq}\right\|d\tau_0
\right)^2d\tau\right]^{1/2}\left[ \nu\int_{s}^{t+s}\big\|n_{\neq}\big\|^{2}_{L^2}d\tau \right]^{1/2}\\
\lesssim& B_1\left(t\lambda_{\nu,\alpha}+1\right)^{1/2}\left(\nu\lambda^{-1}_{\nu,\alpha}\right)^{1/2} C^{1/2}_{\neq}e^{-\lambda_{\nu,\alpha}(s-s_0)/2}\big\|n_0\big\|_{L^2},
\end{aligned}
$$
where we use that for  $s\geq s_0$, one has
\begin{equation}\label{eq:4.4401}
\begin{aligned}
&\left[\nu\int_{s}^{t+s}\left(\int_{\tau}^{t+s}\left\|\Lambda^{\alpha/2}\mathcal{S}_{\tau_0}P_{\neq}\right\|d\tau_0
\right)^2d\tau\right]^{1/2}\\
\lesssim&\left(t/\lambda_{\nu,\alpha}+1/\lambda^2_{\nu,\alpha}\right)^{1/2}e^{-\lambda_{\nu,\alpha}s/2}
\lesssim \left(t\lambda_{\nu,\alpha}+1\right)^{1/2}.
\end{aligned}
\end{equation}
Therefore, we have
\begin{equation}\label{eq:4.44}
\begin{aligned}
&\nu\int_{s}^{t+s}\int_{\tau}^{t+s}\left\|\mathcal{R}_{n_{\neq}\mathbf{B}_1(n^0)}(\tau_0)\right\|_{L^2}d\tau_0d\tau\\
\lesssim &B_1\left(t\lambda_{\nu,\alpha}+1\right)^{1/2}\left(\left(\nu\lambda^{-1}_{\nu,\alpha}\right)^{3/2-2/\alpha}+\left(\nu\lambda^{-1}_{\nu,\alpha}\right)^{1/2}\right) C^{1/2}_{\neq}e^{-\lambda_{\nu,\alpha}(s-s_0)/2}\big\|n_0\big\|_{L^2}.
\end{aligned}
\end{equation}
By Lemma \ref{lem:A.1}, one has
\begin{equation}\label{eq:4.45}
\begin{aligned}
\nu\int_{s}^{t+s}\left\|\mathcal{S}_{t+s-\tau}\left(n_{\neq}\partial_y\mathbf{B}_1(n^0)\right)\right\|_{L^2}d\tau
\lesssim &\nu\int_{s}^{t+s}\big\|n_{\neq}\big\|_{L^2}\big\|n^0\big\|_{L^\infty}d\tau\\
\lesssim &\left(\nu\int_{s}^{t+s} \big\|n_{\neq}\big\|^2_{L^2}d\tau\right)^{1/2}
\left(\nu\int_{s}^{t+s} \big\|n^0\big\|^2_{L^\infty}d\tau\right)^{1/2}\\
\lesssim& \nu B_2\left(\lambda^{-1}_{\nu,\alpha} t\right)^{1/2}C^{1/2}_{\neq}e^{-\lambda_{\nu,\alpha}(s-s_0)/2}\big\|n_0\big\|_{L^2}.
\end{aligned}
\end{equation}
Combining \eqref{eq:4.42}, \eqref{eq:4.44} and \eqref{eq:4.45}, and denote
$$
\begin{aligned}
K_7=&\left((\nu \lambda^{-1}_{\nu,\alpha})^{1-1/\alpha}+(\nu \lambda^{-1}_{\nu,\alpha})^{1/2}\right)B_1\\
&+B_1\left(t\lambda_{\nu,\alpha}+1\right)^{1/2}\left(\left(\nu\lambda^{-1}_{\nu,\alpha}\right)^{3/2-2/\alpha}
+\left(\nu\lambda^{-1}_{\nu,\alpha}\right)^{1/2}\right)+\nu B_2\left(\lambda^{-1}_{\nu,\alpha} t\right)^{1/2},
\end{aligned}
$$
then we have
\begin{equation}\label{eq:4.46}
\nu\int_{s}^{t+s}\left\|\mathcal{S}_{t+s-\tau}\left(\partial_yn_{\neq}\mathbf{B}_1(n^0)\right)\right\|_{L^2}d\tau\\
\lesssim K_7C^{1/2}_{\neq}e^{-\lambda_{\nu,\alpha}(s-s_0)/2}\big\|n_0\big\|_{L^2}.
\end{equation}
Combining \eqref{eq:4.36}-\eqref{eq:4.40} and \eqref{eq:4.46}, we have
$$
\begin{aligned}
\big\|n(t+s)\big\|_{L^2}\leq&e^{-\lambda_{\nu,\alpha} t+\pi/2}\big\|n_{\neq}(s)\big\|_{L^2}+\left(\nu^{6/7}t\right)^{1/2}\left(\left(\lambda^{-1}_{\nu,\alpha} \nu\right)^{1/2}+4\right)C^{1/2}_{\neq}e^{-\lambda_{\nu,\alpha}(s-s_0)/2}\big\|n_{0}\big\|_{L^2}\\
&+\left(\left(\nu^{6/7}t\right)^{1/2}\left(\lambda^{-1}_{\nu,\alpha} \nu\right)^{1/2}
+\left(\nu^{6/7}t\right)^{1/2}+B_2\nu (\lambda^{-1}_{\nu,\alpha} t)^{1/2}\right)C^{1/2}_{\neq}e^{-\lambda_{\nu,\alpha}(s-s_0)/2}\big\|n_0\big\|_{L^2}\\
&+B_2\lambda^{-1}_{\nu,\alpha}\nu C_{\neq}^{1/2}e^{-\lambda_{\nu,\alpha}(s-s_0)/2}\big\|n_0\big\|_{L^2}+ K_7C^{1/2}_{\neq}e^{-\lambda_{\nu,\alpha}(s-s_0)/2}\big\|n_0\big\|_{L^2}.
\end{aligned}
$$
Taking $\tau^{\ast}=8\lambda^{-1}_{\nu,\alpha}$, combining \eqref{eq:1.10}, $\alpha\geq3/2$ and $\nu$ is small, one has
$$
\begin{aligned}
\big\|n_{\neq}(\tau^\ast+s)\big\|_{L^2}\leq \frac{1}{4}e^{-4}\big\|n_{\neq}(s)\big\|_{L^2}+\frac{1}{2}e^{-4}e^{-\lambda_{\nu,\alpha} (s-s_0)/2}\big\|n_0\big\|_{L^2}.
\end{aligned}
$$
Then we can easily get
$$
\big\|n_{\neq}(\tau^\ast+s_0)\big\|_{L^2}\leq e^{-4}\big\|n_0\big\|_{L^2}.
$$
Assume that for $k=m\in \mathbb{Z}^{+}$,  one has
\begin{equation}\label{eq:4.47}
\big\|n_{\neq}(m\tau^\ast+s_0)\big\|_{L^2}\leq e^{-4m}\big\|n_0\big\|_{L^2},
\end{equation}
then for $k=m+1$, we have
$$
\begin{aligned}
\big\|n_{\neq}((m+1)\tau^\ast+s)\big\|_{L^2}&\leq \frac{1}{4}e^{-4}\big\|n_{\neq}(m\tau^\ast+s_0)\big\|_{L^2}+\frac{1}{2}e^{-4}e^{-4m}\big\|n_0\big\|_{L^2}\\
&\leq e^{-4(m+1)}\big\|n_0\big\|_{L^2}.
\end{aligned}
$$
By same argument, we know that for any $s_0\leq t\leq T^{\ast}$, there exist $\tau_0$ and $m\in \mathbb{Z}^{+}$, such that
$$
t=m\tau_\ast+\tau_0+s_0,\ \ \ 0\leq \tau_0\leq \tau^{\ast},
$$
then by the local estimate of \eqref{eq:2.4} and \eqref{eq:4.47}, one has
$$
\big\|n_{\neq}(m\tau^\ast+s_0+\tau_0)\big\|_{L^2}\leq 2\big\|n_{\neq}(m\tau^\ast+s_0)\big\|_{L^2}
\leq2e^{-4m}\big\|n_0\big\|_{L^2}.
$$
Then we obtain
$$
\big\|n_{\neq}(t)\big\|^2_{L^2}\leq 4e^{-8m}\big\|n_0\big\|^2_{L^2}\leq 4e^{8}e^{-\lambda_{\nu,\alpha} (t-s_0)}\big\|n_0\big\|^2_{L^2}\leq 2C_{\neq}e^{-\lambda_{\nu,\alpha} (t-s_0)}\big\|n_0\big\|^2_{L^2},
$$
where $C_{\neq}\geq 2e^{8}$. This completes the proof of Proposition \ref{prop:2.9}.
\end{proof}

\begin{remark}
In this paper, we choose $C_{\neq}\geq 2e^{8}$, and combining local estimate, Assumption \ref{ass:2.8} and Proposition \ref{prop:2.9}, we obtain the global $L^2$ estimate of solution to equation \eqref{eq:1.11}.
\end{remark}

\appendix

\section{Energy estimates}\label{sec.A}

In this section, we establish some useful energy estimate for this paper.

\subsection{Attractive kernel estimate}
The attractive kernel operator $\mathbf{B}_1(n^0)$ and $\mathbf{B}_2(n_{\neq})$ are defined in \eqref{eq:2.5}.
We have following estimates.

\begin{lemma}\label{lem:A.1}
The operator $\mathbf{B}_1(n^0)$ and $\mathbf{B}_2(n_{\neq})$ are defined in \eqref{eq:2.5}, one has
\begin{itemize}
\item [(1)] The estimates of~$\mathbf{B}_1(n^0)$:
$$
\begin{aligned}
\big\|\partial_y\mathbf{B}_1(n^0)\big\|_{L^p}&\lesssim \big\|n^0\big\|_{L^P},\ \ \ \ \ \ \ 1\leq p\leq\infty,\\
\left\|\Lambda^{\beta}_y\mathbf{B}_1(n^0)\right\|_{L^2}&\lesssim \left\|\Lambda^{\beta-1}n^0\right\|_{L^2}, \ \ \ \ \  \beta>1,\\
\big\|\mathbf{B}_1(n^0)\big\|_{L^\infty}+\left\|\Lambda^{\beta}_y\mathbf{B}_1(n^0)\right\|_{L^\infty}&\lesssim \big\|n^0\big\|_{L^2},\ \ \ \ \ \ \ \ \ 0<\beta\leq1/2.
\end{aligned}
$$
\item [(2)]The estimates of~$\mathbf{B}_2(n_{\neq})$
$$
\begin{aligned}
\big\|\nabla\cdot\mathbf{B}_2(n_{\neq})\big\|_{L^p}+\big\|\partial_x\mathbf{B}_2(n_{\neq})\big\|_{L^p}
+\big\|\partial_y\mathbf{B}_2(n_{\neq})\big\|_{L^p}
&\lesssim \big\|n_{\neq}\big\|_{L^p},  \ \ \ 1<p<\infty,\\
\big\|\mathbf{B}_2(n_{\neq})\big\|_{L^\infty}&\lesssim \big\|n_{\neq}\big\|_{L^4}.
\end{aligned}
$$
\end{itemize}
\end{lemma}

\begin{proof}
We finish the proof by following two steps.

\noindent (1) {\em The estimates of~$\mathbf{B}_1(n^0)$}.
According to the definition of $\mathbf{B}_1(n^0)$ in \eqref{eq:2.5}, we have
$$
\big\|\partial_y\mathbf{B}_1(n^0)\big\|_{L^p}=\big\|n^0-\overline{n}\big\|_{L^p}
\lesssim\big\|n^0\big\|_{L^p}+\overline{n}.
$$
and by  H\"{o}lder's inequality, one has
$$
\overline{n}=\frac{1}{4\pi^2}\int_{\mathbb{T}^2}n(t,x,y)dxdy=\frac{1}{2\pi}\int_{\mathbb{T}}n^0(t,y)dy\lesssim \big\|n^0\big\|_{L^p},
$$
thus we have
\begin{equation}\label{eq:A.1}
\big\|\partial_y\mathbf{B}_1(n^0)\big\|_{L^p}\lesssim \big\|n^0\big\|_{L^p}.
\end{equation}
Combining the definition of $\mathbf{B}_1(n^0)$ and $\beta>1$, we can easily get
\begin{equation}
\left\|\Lambda^{\beta}_y\mathbf{B}_1(n^0)\right\|_{L^2}=\left\|\Lambda^{\beta-1}_y\partial_y\mathbf{B}_1(n^0)\right\|_{L^2}
=\left\|\Lambda^{\beta-1}_y(n^0-\overline{n})\right\|_{L^2}\lesssim\left\|\Lambda^{\beta-1}_yn^0\right\|_{L^2}.
\end{equation}\label{eq:A.2}
By Gagliardo-Nirenberg inequality, one has
$$
\big\|\mathbf{B}_1(n^0)\big\|_{L^\infty}\lesssim\big\|\mathbf{B}_1(n^0)\big\|^{1/2}_{L^2}
\big\|\partial_y\mathbf{B}_1(n^0)\big\|^{1/2}_{L^2}
\lesssim\big\|\mathbf{B}_1(n^0)\big\|^{1/2}_{L^\infty}\big\|\partial_y\mathbf{B}_1(n^0)\big\|^{1/2}_{L^2},
$$
then we deduce by \eqref{eq:A.1} that
\begin{equation}\label{eq:A.301}
\big\|\mathbf{B}_1(n^0)\big\|_{L^\infty}\lesssim\big\|\partial_y\mathbf{B}_1(n^0)\big\|_{L^2}\lesssim \left\|n^0\right\|_{L^2}.
\end{equation}
By Gagliardo-Nirenberg inequality, \eqref{eq:A.1}, \eqref{eq:A.301} and $0<\beta\leq 1/2$, one has
\begin{equation}\label{eq:A.401}
\begin{aligned}
\left\|\Lambda^{\beta}_y\mathbf{B}_1(n^0)\right\|_{L^\infty}
&\lesssim\big\|\mathbf{B}_1(n^0)\big\|^{1/3-2\beta/3}_{L^1}\big\|\partial_y \mathbf{B}_1(n^0)\big\|^{2/3+2\beta/3}_{L^2}\\
&\lesssim\big\|\mathbf{B}_1(n^0)\big\|^{1/3-2\beta/3}_{L^\infty}
\big\|\partial_y \mathbf{B}_1(n^0)\big\|^{2/3+2\beta/3}_{L^2}\lesssim \big\|n^0\big\|_{L^2}.
\end{aligned}
\end{equation}

\vskip .05in

\noindent (2) {\em The estimates of~$\mathbf{B}_2(n_{\neq})$}.
Obviously, for $1<p<\infty$, one has
\begin{equation}\label{eq:A.502}
\big\|\nabla\cdot\mathbf{B}_2(n_{\neq})\big\|_{L^p}=\big\|n_{\neq}\big\|_{L^p}.
\end{equation}
Since for $0<p<\infty$, one has
$$
\big\|\partial_x\mathbf{B}_2(n_{\neq})\big\|_{L^p}+\big\|\partial_y\mathbf{B}_2(n_{\neq})\big\|_{L^p}\lesssim \sum_{i,j=1}^2\big\|\partial_i\partial_j(-\Delta)^{-1}n_{\neq}\big\|_{L^p},
$$
and
$$
\big\|\partial_i\partial_j(-\Delta)^{-1}n_{\neq}\big\|_{L^p}=\big\|\mathcal{R}_i\mathcal{R}_jn_{\neq}\big\|_{L^p}
\lesssim \big\|n_{\neq}\big\|_{L^p},
$$
where $\partial_1=\partial_x, \partial_2=\partial_y$, $\mathcal{R}_i=\partial_i\Lambda^{-1}$ is Riesz transforms, we have
\begin{equation}\label{eq:A.601}
\big\|\partial_x\mathbf{B}_2(n_{\neq})\big\|_{L^p}+\big\|\partial_y\mathbf{B}_2(n_{\neq})\big\|_{L^p}
\lesssim \big\|n_{\neq}\big\|_{L^p}.
\end{equation}
By Gagliardo-Nirenberg inequality, one has
$$
\big\|\mathbf{B}_2(n_{\neq})\big\|_{L^\infty}
\lesssim\big\|\mathbf{B}_2(n_{\neq})\big\|^{1/2}_{L^2}\big\|\nabla\cdot\mathbf{B}_2(n_{\neq})\big\|^{1/2}_{L^4}
\lesssim\big\|\mathbf{B}_2(n_{\neq})\big\|^{1/2}_{L^\infty}\big\|\nabla\cdot\mathbf{B}_2(n_{\neq})\big\|^{1/2}_{L^4},
$$
then we deduce by \eqref{eq:A.502} that
\begin{equation}\label{eq:A.701}
\big\|\mathbf{B}_2(n_{\neq})\big\|_{L^\infty}\lesssim\big\|\nabla\cdot\mathbf{B}_2(n_{\neq})\big\|_{L^4}
\lesssim \big\|n_{\neq}\big\|_{L^4}.
\end{equation}

Combining \eqref{eq:A.1}-\eqref{eq:A.701}, we finish the proof of Lemma \ref{lem:A.1}.
\end{proof}

\begin{remark}
In the proof of Lemma \ref{lem:A.1}, we used the properties of Riesz transforms. Namely, the Riesz transforms is bounded in $L^p$ for $1<p<\infty$, the details can see \cite[Corollary 4.2.8]{Grafakos.2008}.
\end{remark}

\subsection{Local estimate}\label{App:A.2}
The local estimate makes the Assumption \ref{ass:2.8} reasonable, which is the premise of bootstrap argument. Here we give the local estimate of solution of equation \eqref{eq:2.4}.

\begin{lemma}[Local estimate of $n_{\neq}$]\label{lem:A.301}
Let $\alpha\in[3/2,2)$ and $n_{\neq}$ be a solution of equation \eqref{eq:2.4} with initial data $n_0\geq0, n_0\in H^{\gamma}(\mathbb{T}^2)\cap L^{1}(\mathbb{T}^2), \gamma>1+\alpha$. Assume that the $u(y)=\cos y$ is Kolmogorov flow. Then there exist time $s_0>0, t^\ast=s_0+8\lambda^{-1}_{\nu,\alpha}$, $\nu$ is small enough and a  positive constant $C_{\neq}$, such that for any $s_0\leq s\leq t\leq t^\ast$, one has
$$
\nu\int_{s}^{t}\left\|\Lambda^{\alpha/2} n_{\neq}\right\|^2_{L^2}d\tau\leq 16C_{\ne}e^{-\lambda_{\nu,\alpha}(s-s_0)}\big\|n_0\big\|_{L^2}^2,
$$
and for any $s_0\leq t\leq t^\ast$, one has
$$
\big\|n_{\neq}(t)\big\|^2_{L^2}\leq 4C_{\ne}e^{-\lambda_{\nu,\alpha}(t-s_0)}\big\|n_0\big\|^2_{L^2},
$$
where $\lambda_{\nu,\alpha}$ is defined in \eqref{eq:1.10}.
\end{lemma}

\begin{proof}
Let us multiply both sides of \eqref{eq:1.11} by $n$ and integrate over $\mathbb{T}^2$, to obtain
\begin{equation}\label{eq:A.3}
\frac{1}{2}\frac{d}{dt}\big\|n\big\|^2_{L^2}+\nu\left\|\Lambda^{\alpha/2}n\right\|^2_{L^2}
+\nu\int_{\mathbb{T}^2}\nabla\cdot\left(n\mathbf{B}(n)\right)ndxdy.
\end{equation}
Using integral by part, Lemma \ref{lem:A.1} and energy estimate, one has
$$
\left|\int_{\mathbb{T}^2}\nabla\cdot\left(n\mathbf{B}(n)\right)ndxdy\right|
=\left|\frac{1}{2}\int_{\mathbb{T}^2}\nabla\cdot \mathbf{B}(n)n^2dxdy\right|
\leq \big\|n\big\|^3_{L^3}+\overline{n}\big\|n\big\|^2_{L^2},
$$
and 
$$
\big\|n\big\|^3_{L^3}\lesssim \left\|\Lambda^{\alpha/2}n\right\|^{2/\alpha}_{L^2}\big\|n\big\|^{3-2/\alpha}_{L^2}
\leq \frac{1}{2}\left\|\Lambda^{\alpha/2}n\right\|_{L^2}^2+C\big\|n\big\|_{L^2}^{\frac{3\alpha-2}{\alpha-1}}.
$$
The \eqref{eq:A.3} can be written as
\begin{equation}\label{eq:A.4}
\frac{d}{dt}\big\|n\big\|_{L^2}^2+\nu\left\|\Lambda^{\alpha/2}n\right\|_{L^2}^2\leq C\nu\big\|n\big\|_{L^2}^{\frac{3\alpha-2}{\alpha-1}}+2\nu\overline{n}\big\|n\big\|_{L^2}^2.
\end{equation}
Then we deduce by \eqref{eq:A.4} that there exists
\begin{equation}\label{eq:A.501}
t_0=\frac{\alpha-1}{\overline{n}\alpha\nu}\ln\left(\frac{4^{\frac{\alpha}{2\alpha-2}}\left(2\overline{n}
+C\big\|n_0\big\|_{L^2}^{\frac{\alpha}{2\alpha-2}}\right)}{2\overline{n}+4^{\frac{\alpha}{2\alpha-2}}
C\big\|n_0\big\|_{L^2}^{\frac{\alpha}{2\alpha-2}}}\right)=O(1/\nu),
\end{equation}
such that for any $0\leq t\leq t_0$, one has
\begin{equation}\label{eq:A.1101}
\big\|n\big\|^2_{L^2}\leq 4\big\|n_0\big\|^2_{L^2}.
\end{equation}
Taking $s_0=t_0/2$ and $t^\ast=s_0+8\lambda^{-1}_{\nu,\alpha}$, we know that for $\nu$ is small enough, one has
$$
s_0\leq t^\ast\leq t_0.
$$
Then for any $s_0\leq t\leq t^\ast$, one has
\begin{equation}\label{eq:A.5}
\big\|n^0\big\|^2_{L^2}+\big\|n_{\neq}\big\|^2_{L^2}\leq 4\big\|n_0\big\|^2_{L^2}\leq 4C_{\ne}e^{-2N_0}\big\|n_0\big\|_{L^2}^2\leq 4C_{\ne}e^{-\lambda_{\nu,\alpha}(t-s_0)}\big\|n_0\big\|^2_{L^2},
\end{equation}
where $C_{\ne}\geq e^{8}$. Combining \eqref{eq:A.4} and \eqref{eq:A.5}, we imply that for any $s_0\leq s\leq t\leq t^\ast$ and $\nu$ is small, one has
\begin{equation}\label{eq:A.6}
\begin{aligned}
\nu\int_{s}^{t}\left\|\Lambda^{\alpha/2}n\right\|_{L^2}^2d\tau
&\leq C\nu\int_{s}^{t}\big\|n\big\|_{L^2}^{\frac{3\alpha-2}{\alpha-1}}d\tau+2\nu\overline{n}\int_{s}^{t}\big\|n\big\|_{L^2}^2d\tau
+\big\|n(s)\big\|_{L^2}^2\\
&\leq  C\nu\left(4C_{\neq}\big\|n_0\big\|^2_{L^2}\right)^{\frac{3\alpha-2}{2\alpha-2}}\int_{s}^{t}e^{-\lambda_{\nu,\alpha} \tau}d\tau e^{\lambda_{\nu,\alpha} s_0}\\
&\ \ \ +8\overline{n}\nu C_{\neq}\big\|n_0\big\|^2_{L^2}\int_{s}^{t}e^{-\lambda_{\nu,\alpha} \tau}d\tau e^{\lambda_{\nu,\alpha} s_0}+4C_{\ne}e^{-\lambda_{\nu,\alpha}s}\big\|n_0\big\|_{L^2}^2e^{\lambda_{\nu,\alpha} s_0}\\
&\leq C\nu\left(4C_{\neq}\big\|n_0\big\|^2_{L^2}\right)^{\frac{3\alpha-2}{2\alpha-2}}\lambda^{-1}_{\nu,\alpha} e^{-\lambda_{\nu,\alpha} s}e^{\lambda_{\nu,\alpha} s_0}\\
&\ \ \ +8\overline{n}\nu C_{\neq}\big\|n_0\big\|^2_{L^2}\lambda_{\nu,\alpha}^{-1}e^{-\lambda_{\nu,\alpha} s}e^{\lambda_{\nu,\alpha} s_0}+4C_{\ne}e^{-\lambda_{\nu,\alpha}s}\big\|n_0\big\|_{L^2}^2e^{\lambda_{\nu,\alpha} s_0}\\
&\leq 16C_{\ne}e^{-\lambda_{\nu,\alpha}s}\big\|n_0\big\|_{L^2}^2e^{\lambda_{\nu,\alpha} s_0}.
\end{aligned}
\end{equation}
Since
\begin{equation}\label{eq:A.1401}
\left\|\Lambda^{\alpha/2}n\right\|_{L^2}^2=\left\|\Lambda^{\alpha/2}n_{\neq}\right\|_{L^2}^2+\left\|\Lambda_y^{\alpha/2} n^0\right\|^2_{L^2_y},
\end{equation}
we deduce by \eqref{eq:A.6} that
\begin{equation}\label{eq:A.7}
\nu\int_{s}^{t}\left\|\Lambda^{\alpha/2}n_{\neq}\right\|_{L^2}^2d\tau
\leq\nu\int_{s}^{t}\left\|\Lambda^{\alpha/2}n\right\|_{L^2}^2d\tau
\leq 16C_{\ne}e^{-\lambda_{\nu,\alpha}(s-s_0)}\big\|n_0\big\|_{L^2}^2.
\end{equation}
Combining \eqref{eq:A.5} and \eqref{eq:A.7}, we finish the proof of Lemma \ref{lem:A.301}.
\end{proof}

\begin{remark}
Based on the \eqref{eq:A.4}, \eqref{eq:A.1101} and \eqref{eq:A.1401}, we deduce that for any $0\leq t\leq s_0$, one has
\begin{equation}\label{eq:A.1601}
\big\|n_{\neq}\big\|_{L^2}\leq\big\|n\big\|_{L^2}\leq2\big\|n_0\big\|_{L^2},
\end{equation}
and for any $0\leq s\leq t\leq s_0$, one has
\begin{equation}\label{eq:A.1602}
\nu\int_{s}^{t}\left\|\Lambda^{\alpha/2}n_{\neq}\right\|_{L^2}^2d\tau
\leq\nu\int_{s}^{t}\left\|\Lambda^{\alpha/2}n\right\|_{L^2}^2d\tau
\lesssim\left( \big\|n_0\big\|^{\frac{\alpha}{\alpha-1}}_{L^2}+\overline{n}\right)\big\|n_0\big\|^2_{L^2}.
\end{equation}
\end{remark}

\subsection{High order estimate}
Here we establish the $H^1$ estimate of $n$ to equation \eqref{eq:1.11} in the case of $n\in L^1\cap L^\infty$, it is as follows

\begin{lemma}\label{lem:A.4}
Let $\alpha\in [3/2,2)$ and $n$ be the solution of \eqref{eq:1.11} with initial data $n_0\geq0, n_0\in H^\gamma(\mathbb{T}^2)\cap L^1(\mathbb{T}^2), \gamma>1+\alpha$. Assume that the $u(y)=\cos y$ is Kolmogorov flow. If $n\in L^1\cap L^\infty$, then
$$
\big\|n\big\|_{H^1}\lesssim \nu^{-1/14}.
$$
\end{lemma}

\begin{proof}
Applying $\partial_x$ to \eqref{eq:1.11} and multiplying both sides by $\partial_x n$ and integrating over $\mathbb{T}^2$, we obtain
\begin{equation}\label{eq:A.8}
\frac{1}{2}\frac{d}{dt}\big\|\partial_x n\big\|^2_{L^2}+\nu\left\|\Lambda^{\alpha/2}\partial_x n\right\|^2_{L^2}
+\nu\int_{\mathbb{T}^2}\partial_x\left(\nabla\cdot(n \mathbf{B}(n))\right)\partial_x ndxdy=0.
\end{equation}
Since
$$
\int_{\mathbb{T}^2}\partial_x\left(\nabla\cdot(n \mathbf{B}(n))\right)\partial_x ndxdy=\int_{\mathbb{T}^2}\partial_x\left(\nabla n \cdot\mathbf{B}(n)\right)\partial_x ndxdy
+\int_{\mathbb{T}^2}\partial_x\left(n \nabla\cdot\mathbf{B}(n)\right)\partial_x ndxdy,
$$
and one has from Lemma \ref{lem:2.4} and \ref{lem:2.5} that
$$
\left|\int_{\mathbb{T}^2}\partial_x\left(\nabla n \cdot\mathbf{B}(n)\right)\partial_x ndxdy\right|
\lesssim \left\|\Lambda^{1-\alpha/2}\left(\nabla n \cdot\mathbf{B}(n)\right)\right\|_{L^2}\left\|\Lambda^{1+\alpha/2}n\right\|_{L^2},
$$
and
$$
\left\|\Lambda^{1-\alpha/2}\left(\nabla n \cdot\mathbf{B}(n)\right)\right\|_{L^2}
\lesssim \left\|\Lambda^{1-\alpha/2}\nabla n\right\|_{L^2}\big\|\mathbf{B}(n)\big\|_{L^\infty}
+\big\|\nabla n\big\|_{L^2}\left\|\Lambda^{1-\alpha/2}\mathbf{B}(n)\right\|_{L^\infty}.
$$
Combining
$$
\left\|\Lambda^{1-\alpha/2}\nabla n\right\|_{L^2}\lesssim\left\|\Lambda^{2-\alpha/2} n\right\|_{L^2}
\lesssim\left\|\Lambda^{1+\alpha/2}n\right\|^{2/\alpha-1}_{L^2}\big\|\Lambda n\big\|^{2-2/\alpha}_{L^2},
$$
$$
\big\|\mathbf{B}(n)\big\|_{L^\infty}\lesssim \big\|n\big\|_{L^4}\lesssim1,
$$
$$
\left\|\Lambda^{1-\alpha/2}\mathbf{B}(n)\right\|_{L^\infty}
\lesssim\big\|\mathbf{B}(n)\big\|^{\alpha/2-1/2}_{L^4}\big\|\Lambda\mathbf{B}(n)\big\|^{3/2-\alpha/2}_{L^4}
\lesssim \big\|n\big\|_{L^4}\lesssim1,
$$
we obtain
$$
\begin{aligned}
&\left\|\Lambda^{1-\alpha/2}\left(\nabla n \cdot\mathbf{B}(n)\right)\right\|_{L^2}\left\|\Lambda^{1+\alpha/2}n\right\|_{L^2}\\
\lesssim&\left\|\Lambda^{1+\alpha/2}n\right\|^{2/\alpha}_{L^2}\big\|\Lambda n\big\|^{2-2/\alpha}_{L^2}
+\left\|\Lambda^{1+\alpha/2}n\right\|_{L^2}\big\|\Lambda n\big\|_{L^2}\\
\leq& \frac{1}{4}\left\|\Lambda^{1+\alpha/2}n\right\|^2_{L^2}+C\big\|\Lambda n\big\|^2_{L^2}.
\end{aligned}
$$
and
$$
\left|\int_{\mathbb{T}^2}\partial_x\left(n \nabla\cdot\mathbf{B}(n)\right)\partial_x ndxdy \right|
\lesssim\big\|n\big\|_{L^\infty}\big\|\partial_x n\big\|^2_{L^2}\lesssim \big\|\Lambda n\big\|^2_{L^2}.
$$
Then one gets
\begin{equation}\label{eq:A.9}
\left|\nu\int_{\mathbb{T}^2}\partial_x\left(\nabla\cdot(n \mathbf{B}(n))\right)\partial_x ndxdy\right|
\leq\frac{\nu}{4}\left\|\Lambda^{1+\alpha/2}n\right\|^2_{L^2}+C\nu\big\|\Lambda n\big\|^2_{L^2}.
\end{equation}
Combining \eqref{eq:A.8} and \eqref{eq:A.9}, we have
\begin{equation}\label{eq:A.10}
\frac{1}{2}\frac{d}{dt}\big\|\partial_x n\big\|^2_{L^2}+\nu\left\|\Lambda^{\alpha/2}\partial_x n\right\|^2_{L^2}
\leq \frac{\nu}{4}\left\|\Lambda^{1+\alpha/2}n\right\|^2_{L^2}+C\nu\big\|\Lambda n\big\|^2_{L^2}.
\end{equation}
Applying $\partial_y$ to \eqref{eq:1.11}, multiplying both sides by $\partial_y n$ and integrating over $\mathbb{T}^2$, we obtain
\begin{equation}\label{eq:A.11}
\begin{aligned}
\frac{1}{2}\frac{d}{dt}\big\|\partial_y n\big\|^2_{L^2}+\nu\left\|\Lambda^{\alpha/2}\partial_y n\right\|^2_{L^2}
&+\int_{\mathbb{T}^2}\partial_y(u(y) \partial_{x}n)\partial_y ndxdy\\
&+\nu\int_{\mathbb{T}^2}\partial_y\left(\nabla\cdot(n \mathbf{B}(n))\right)\partial_y ndxdy=0.
\end{aligned}
\end{equation}
Obviously, one has
$$
\left|\int_{\mathbb{T}^2}\partial_y\left(u(y) \partial_{x}n\right)\partial_y ndxdy\right|
=\left|\int_{\mathbb{T}^2}u'(y)\partial_{x}n\partial_y ndxdy\right|
\lesssim \big\|\Lambda n\big\|^2_{L^2}.
$$
Similar to \eqref{eq:A.9}, we obtain
$$
\left|\nu\int_{\mathbb{T}^2}\partial_y\left(\nabla\cdot(n \mathbf{B}(n))\right)\partial_y ndxdy\right|
\leq\frac{\nu}{4}\left\|\Lambda^{1+\alpha/2}n\right\|^2_{L^2}+C\nu\big\|\Lambda n\big\|^2_{L^2}.
$$
Then we deduce by \eqref{eq:A.11} that
\begin{equation}\label{eq:A.12}
\frac{1}{2}\frac{d}{dt}\big\|\partial_y n\big\|^2_{L^2}+\nu\left\|\Lambda^{\alpha/2}\partial_y n\right\|^2_{L^2}
\leq \frac{\nu}{4}\left\|\Lambda^{1+\alpha/2}n\right\|^2_{L^2}+C\nu\big\|\Lambda n\big\|^2_{L^2}+C\big\|\Lambda n\big\|^2_{L^2}.
\end{equation}
Combining \eqref{eq:A.10} and \eqref{eq:A.12}, we have
\begin{equation}\label{eq:A.13}
\frac{d}{dt}\big\|\Lambda n\big\|^2_{L^2}+\nu\left\|\Lambda^{\alpha/2+1} n\right\|^2_{L^2}
\lesssim\nu\big\|\Lambda n\big\|^{2}_{L^2}+\big\|\Lambda n\big\|^2_{L^2}.
\end{equation}
Since $\|n\|_{L^1}+\|n\|_{L^\infty}\lesssim1$, one has
$$
-\left\|\Lambda^{\alpha/2+1} n\right\|^2_{L^2}\lesssim -\big\|\Lambda n\big\|^{16}_{L^2}.
$$
Therefore, the \eqref{eq:A.13} is written as
$$
\frac{d}{dt}\big\|\Lambda n\big\|^2_{L^2}\lesssim -\nu\big\|\Lambda n\big\|^{16}_{L^2}
+\nu\big\|\Lambda n\big\|^{2}_{L^2}+2\big\|\Lambda n\big\|^2_{L^2},
$$
Then we imply
$$
\big\|\Lambda n\big\|_{L^2}\lesssim \nu^{-1/14}.
$$
This completes the proof of Lemma \ref{lem:A.4}
\end{proof}

\section{Semigroup theory}\label{sec.B}

In this section, we establish some useful semigroup theory, define semigroup operator
\begin{equation}\label{eq:B.1}
\mathcal{S}_t=e^{-t\mathcal{L}_{\nu,\alpha}}, \ \
\end{equation}
where the $\mathcal{L}_{\nu,\alpha}$ is defined in \eqref{eq:1.8} and shear flow $u(y)$ satisfies the assumption in Theorem \ref{thm:1.1}. We give the following two lemmas.

\begin{lemma}\label{lem:B.1}
Let $\mathcal{S}_t$ be defined in \eqref{eq:B.1}, then for $0\leq s\leq t<\infty$, one has
$$
\nu\int_s^t\left\|\Lambda^{\alpha/2}\mathcal{S}_{\tau}P_{\neq}\right\|^2d\tau\lesssim e^{-2\lambda'_{\nu,\alpha }s},\ \ \ \ \nu\int_s^t\left\|\Lambda^{\alpha/2}\mathcal{S}_{\tau}\right\|^2d\tau\lesssim 1,
$$
where $\lambda'_{\nu,\alpha}=\epsilon_0 \nu^{\frac{m}{m+\alpha}}$, $\alpha\in (0,2)$ and
$$
\left\|\Lambda^{\alpha/2}\mathcal{S}_tP_{\neq}\right\|=\sup_{\|g_0\|_{L^2}\ne0}
\frac{\left\|\Lambda^{\alpha/2}\mathcal{S}_tP_{\neq} g_0\right\|_{L^2}}{\|g_0\|_{L^2}},\ \ \
\left\|\Lambda^{\alpha/2}\mathcal{S}_t\right\|=\sup_{\|g_0\|_{L^2}\ne0}\frac{\left\|\Lambda^{\alpha/2}\mathcal{S}_t g_0\right\|_{L^2}}{\|g_0\|_{L^2}}.
$$
\end{lemma}

\begin{proof}
Consider the equation \eqref{eq:1.6}, \eqref{eq:1.7} and \eqref{eq:1.9}, one has
\begin{equation}\label{eq:B.201}
\partial_tg_{\neq}+u(y)\partial_xg_{\neq}+\nu(-\Delta)^{\alpha/2}g_{\neq}=0,
\end{equation}
and the $g_{\neq}$ can be written as
$$
g_{\neq}(t,x,y)=\mathcal{S}_{t}P_{\neq}g_0(x,y).
$$
Then we deduce by Theorem \ref{thm:1.1} and \eqref{eq:B.201} that for any $0\leq s\leq t<\infty$, one has
\begin{equation}\label{eq:B.301}
\big\|g_{\neq}(t,\cdot)\big\|^2_{L^2}+\nu\int_{s}^{t}\left\|\Lambda^{\alpha/2}g_{\neq}(\tau,\cdot)\right\|^2_{L^2}d\tau
\lesssim \big\|g_{\neq}(s,\cdot)\big\|^2_{L^2}\lesssim e^{-2\lambda'_{\nu,\alpha} s}\big\|g_0\big\|^2_{L^2}.
\end{equation}
Thus, we have
\begin{equation}\label{eq:B.3}
\big\|g_{\neq}(t,\cdot)\big\|^2_{L^2}\lesssim e^{-2\lambda'_{\nu,\alpha} t}\big\|g_0\big\|^2_{L^2},\ \ \ \ \nu\int_{s}^{t}\left\|\Lambda^{\alpha/2}\mathcal{S}_\tau P_{\neq} g_0\right\|^2_{L^2}d\tau\lesssim e^{-2\lambda'_{\nu,\alpha} s}\big\|g_0\big\|^2_{L^2}.
\end{equation}
Considering the operator $\Lambda^{\alpha/2} \mathcal{S}_{t}P_{\neq}$, obviously,
\begin{equation}\label{eq:B.4}
\left\|\Lambda^{\alpha/2}\mathcal{S}_tP_{\neq} g_0\right\|_{L^2}\leq \left\|\Lambda^{\alpha/2}\mathcal{S}_tP_{\neq} \right\|\big\|g_0\big\|_{L^2},\ \ \ \left\|\Lambda^{\alpha/2}\mathcal{S}_tP_{\neq}\right\|^2=\sup_{\|g_0\|\ne0}\frac{\left\|\Lambda^{\alpha/2}\mathcal{S}_t P_{\neq}g_0\right\|^2}{\|g_0\|_{L^2}^2},
\end{equation}
then we have
$$
\nu\int_s^t\left\|\Lambda^{\alpha/2}\mathcal{S}_{\tau}P_{\neq}\right\|^2d\tau
=\nu\sup_{\|g_0\|_{L^2}\ne0}\int_s^t\left\|\Lambda^{\alpha/2}\mathcal{S}_tP_{\neq} g_0\right\|^2_{L^2}\big\|g_0\big\|^{-2}_{L^2}d\tau
\lesssim e^{-2\lambda'_{\nu,\alpha} s}.
$$
Similarly, one has
$$
\nu\int_s^t\left\|\Lambda^{\alpha/2}\mathcal{S}_{\tau}\right\|^2d\tau\lesssim 1.
$$
This completes the proof of Lemma \ref{lem:B.1}.
\end{proof}

\begin{lemma}\label{lem:B.4}
Let $\mathcal{S}_t$ be defined in \eqref{eq:B.1} and $f_0\in C^\infty(\mathbb{T}^2)$, we have
$$
\mathcal{S}_t\partial_yf_0=-\Lambda^{\alpha/2}_{y}\mathcal{S}_{t}\Lambda^{-\alpha/2}_{y}\partial_{y}f_0
-\int_{0}^{t}e^{-(t-\tau)\mathcal{L}_{\nu,\alpha}}\mathcal{R}_{f_0}(\tau,x,y)d\tau,
$$
where
$$
\mathcal{R}_{f_0}(\tau,x,y)=C_{\alpha}\sum_{k\in\mathbb{Z}}P.V.\int_{\mathbb{T}}\mathcal{K}(y,\widetilde{y})
\Lambda_{x}^{\alpha/2-1}\partial_{x}\mathcal{S}_{\tau}\left(\Lambda^{1-\alpha/2}_{x}\Lambda^{-\alpha/2}_{\widetilde{y}}
\partial_{\widetilde{y}}f_0\right)d\widetilde{y},
$$
and
$$
\mathcal{K}(y,\widetilde{y})=\frac{u(y)-u(\widetilde{y})}{|y-\widetilde{y}+k|^{1+\alpha/2}}.
$$
\end{lemma}

In order to prove the Lemma \ref{lem:B.4}, we consider the following linear equation
\begin{equation}\label{eq:B.6001}
\partial_tf+u(y)\partial_x f+\nu(-\Delta)^{\alpha/2}f=0,\ \ \ \ f(0,x,y)=f_0(x,y).
\end{equation}
If $\mathbb{G}(t,x_0,y_0,x,y)$ is fundamental solution of equation \eqref{eq:B.6001} on $\mathbb{T}^2$, which satisfies
\begin{equation}\label{eq:B.5}
\begin{cases}
\partial_t\mathbb{G}+u(y)\partial_x \mathbb{G}+\nu(-\Delta)^{\alpha/2}\mathbb{G}=0,\\
\lim\limits_{t\rightarrow 0}\mathbb{G}(t,x_0,y_0,x,y)=\delta(x-x_0,y-y_0),
\end{cases}
\end{equation}
where $\delta(x-x_0,y-y_0)$ denotes the point mass delta function at $(x_0,y_0)$. Then the solution of \eqref{eq:B.6001} can be written as
\begin{equation}\label{eq:B.601}
f(t,x,y)=\mathcal{S}_t f_0=\int_{\mathbb{T}^2}\mathbb{G}(t,x_0,y_0,x,y)f_0(x_0,y_0)dx_0dy_0.
\end{equation}
Combining Lemma \ref{lem:2.5}, one gets
\begin{equation}\label{eq:B.6}
\begin{aligned}
\mathcal{S}_t \partial_yf_0&=\int_{\mathbb{T}^2}\mathbb{G}(t,x_0,y_0,x,y)\partial_{y_0}f_0(x_0,y_0)dx_0dy_0\\
&=\int_{\mathbb{T}^2}\Lambda^{\alpha/2}_{y_0}\mathbb{G}(t,x_0,y_0,x,y)\Lambda^{-\alpha/2}_{y_0}\partial_{y_0}f_0(x_0,y_0)dx_0dy_0.
\end{aligned}
\end{equation}
Applying the operator $\Lambda^{\alpha/2}_{y_0}$ and $\Lambda^{\alpha/2}_{y}$ to \eqref{eq:B.5}, then we have
\begin{equation}\label{eq:B.7}
\begin{cases}
\partial_t \left(\Lambda^{\alpha/2}_{y}\mathbb{G}+\Lambda^{\alpha/2}_{y_0}\mathbb{G}\right)
+u(y)\partial_x \left(\Lambda^{\alpha/2}_{y}\mathbb{G}+\Lambda^{\alpha/2}_{y_0}\mathbb{G}\right)\\
\ \ \ \ \ \ \ \ \ \ \ \ \ \ \ \ \ \ \ +\nu (-\Delta)^{\alpha/2} \left(\Lambda^{\alpha/2}_{y}\mathbb{G}+\Lambda^{\alpha/2}_{y_0}\mathbb{G}\right)+\mathcal{R}(t,x_0,y_0,x,y)=0,\\
\left(\Lambda^{\alpha/2}_{y}\mathbb{G}+\Lambda^{\alpha/2}_{y_0}\mathbb{G}\right)
(0,x_0,y_0,x,y)=0,
\end{cases}
\end{equation}
where
\begin{equation}\label{eq:B.8}
\mathcal{R}(t,x_0,y_0,x,y)=C_{\alpha}\sum_{k\in\mathbb{Z}}
P.V.\int_{\mathbb{T}}\frac{\left(u(y)-u(\widetilde{y})\right)\partial_x\mathbb{G}
(t,x_0,y_0,x,\widetilde{y})}{|y-\widetilde{y}+k|^{1+\alpha/2}}d\widetilde{y}.
\end{equation}

\begin{remark}\label{eq:rem:B.2}
In \eqref{eq:B.7}, we use the following equality as
$$
\begin{aligned}
\Lambda^{\alpha/2}_{y}\left(u(y)\partial_x \mathbb{G}\right)
=&C_{\alpha}\sum_{k\in\mathbb{Z}}P.V.\int_{\mathbb{T}}\frac{u(y)\partial_x \mathbb{G}(t,x_0,y_0,x,y)-u(\widetilde{y})\partial_x \mathbb{G}(t,x_0,y_0,x,\widetilde{y})}{|y-\widetilde{y}+k|^{1+\alpha/2}}d\widetilde{y}\\
=&C_{\alpha}\sum_{k\in\mathbb{Z}}P.V.\int_{\mathbb{T}}\frac{u(y)\partial_x \mathbb{G}(t,x_0,y_0,x,y)-u(y)\partial_x \mathbb{G}(t,x_0,y_0,x,\widetilde{y})}{|y-\widetilde{y}+k|^{1+\alpha/2}}d\widetilde{y}\\
&+C_{\alpha}\sum_{k\in\mathbb{Z}}P.V.\int_{\mathbb{T}}\frac{u(y)\partial_x \mathbb{G}(t,x_0,y_0,x,\widetilde{y})-u(\widetilde{y})\partial_x \mathbb{G}(t,x_0,y_0,x,\widetilde{y})}{|y-\widetilde{y}+k|^{1+\alpha/2}}d\widetilde{y}\\
=&u(y)C_{\alpha}\sum_{k\in\mathbb{Z}}P.V.\int_{\mathbb{T}}\frac{\partial_x \mathbb{G}(t,x_0,y_0,x,y)-\partial_x \mathbb{G}(t,x_0,y_0,x,\widetilde{y})}{|y-\widetilde{y}+k|^{1+\alpha/2}}d\widetilde{y}\\
&+C_{\alpha}\sum_{k\in\mathbb{Z}}
P.V.\int_{\mathbb{T}}\frac{\left(u(y)-u(\widetilde{y})\right)\partial_x\mathbb{G}
(t,x_0,y_0,x,\widetilde{y})}{|y-\widetilde{y}+k|^{1+\alpha/2}}d\widetilde{y}\\
=&u(y)\partial_x\Lambda^{\alpha/2}_y\mathbb{G}(t,x_0,y_o,x,y)+\mathcal{R}(t,x_0,y_0,x,y).
\end{aligned}
$$
\end{remark}

\begin{remark}
Since for any $g\in C^\infty(\mathbb{T}^2)$, one has
$$
\begin{aligned}
\lim_{t\rightarrow 0}\int_{\mathbb{T}^2}\Lambda^{\alpha/2}_{y}\mathbb{G}(t,x_0,y_0,x,y)g(x_0,y_0)dx_0dy_0
=&\Lambda^{\alpha/2}_{y}\int_{\mathbb{T}^2}\delta(x-x_0,y-y_0)g(x_0,y_0)dx_0dy_0\\
=&\int_{\mathbb{T}^2}\delta(x_0,y_0)\Lambda^{\alpha/2}_{y}g(x-x_0,y-y_0)dx_0dy_0,
\end{aligned}
$$
and
$$
\begin{aligned}
\lim_{t\rightarrow 0}\int_{\mathbb{T}^2}\Lambda^{\alpha/2}_{y_0}\mathbb{G}(t,x_0,y_0,x,y)g(x_0,y_0)dx_0dy_0
=&\int_{\mathbb{T}^2}\delta(x-x_0,y-y_0)\Lambda^{\alpha/2}_{y_0}g(x_0,y_0)dx_0dy_0\\
=&\int_{\mathbb{T}^2}\delta(x_0,y_0)\Lambda^{\alpha/2}_{y_0}g(x-x_0,y-y_0)dx_0dy_0.
\end{aligned}
$$
According to the definition of $\Lambda^{\alpha/2}_{y}$ and $\Lambda^{\alpha/2}_{y_0}$ in \eqref{eq:2.2}, one has
$$
\Lambda^{\alpha/2}_{y}g(x-x_0,y-y_0)+\Lambda^{\alpha/2}_{y_0}g(x-x_0,y-y_0)=0,
$$
then we can get
$$
\lim_{t\rightarrow 0}\int_{\mathbb{T}^2}\left(\Lambda^{\alpha/2}_{y}\mathbb{G}+\Lambda^{\alpha/2}_{y_0}\mathbb{G}\right)
(t,x_0,y_0,x,y)g(x_0,y_0)dx_0dy_0=0.
$$
Thus we have
$$
\lim_{t\rightarrow 0}\left(\Lambda^{\alpha/2}_{y}\mathbb{G}+\Lambda^{\alpha/2}_{y_0}\mathbb{G}\right)
(0,x_0,y_0,x,y)=0.
$$
\end{remark}

Combining Duhamel's principle, \eqref{eq:1.8} and \eqref{eq:B.7},we have
\begin{equation}\label{eq:B.10}
\Lambda^{\alpha/2}_{y_0}\mathbb{G}(t,x_0,y_0,x,y)=-\Lambda^{\alpha/2}_{y}\mathbb{G}(t,x_0,y_0,x,y)
-\int_{0}^{t}e^{-(t-\tau)\mathcal{L}_{\nu,\alpha}}\mathcal{R}(\tau,x_0,y_0,x,y)d\tau.
\end{equation}
By similar argument, we can easily get
\begin{equation}\label{eq:B.1101}
\partial_x\mathbb{G}(t,x_0,y_0,x,y)=-\partial_{x_0}\mathbb{G}(t,x_0,y_0,x,y),
\end{equation}
\begin{equation}\label{eq:B.1201}
 \Lambda^{\alpha/2}_{x}\mathbb{G}(t,x_0,y_0,x,y)=-\Lambda^{\alpha/2}_{x_0}\mathbb{G}(t,x_0,y_0,x,y).
\end{equation}
and
\begin{equation}\label{eq:B.1301}
\Lambda^{-\alpha/2}_{x}\partial_x \mathbb{G}(t,x_0,y_0,x,y)=- \Lambda^{-\alpha/2}_{x_0}\partial_{x_0}\mathbb{G}(t,x_0,y_0,x,y),\ \ \ 0<\alpha<2.
\end{equation}
\vskip .05in

Next, we give the proof of Lemma \ref{lem:B.4}.

\begin{proof}[The proof of Lemma \ref{lem:B.4}]
Combining \eqref{eq:B.6} and \eqref{eq:B.10}, one has
\begin{equation}\label{eq:B.11}
\begin{aligned}
\mathcal{S}_t\partial_yf_0&=\int_{\mathbb{T}^2}\Lambda^{\alpha/2}_{y_0}\mathbb{G}(t,x_0,y_0,x,y)
\Lambda^{-\alpha/2}_{y_0}\partial_{y_0}f_0(x_0,y_0)dx_0dy_0\\
&=-\int_{\mathbb{T}^2}\Lambda^{\alpha/2}_{y}\mathbb{G}(t,x_0,y_0,x,y)
\Lambda^{-\alpha/2}_{y_0}\partial_{y_0}f_0(x_0,y_0)dx_0dy_0\\
&\ \ \ -\int_{\mathbb{T}^2}\left(\int_{0}^{t}e^{-(t-\tau)\mathcal{L}_{\nu,\alpha}}\mathcal{R}(\tau,x_0,y_0,x,y)d\tau\right)
\Lambda^{-\alpha/2}_{y_0}\partial_{y_0}f_0(x_0,y_0)dx_0dy_0.
\end{aligned}
\end{equation}
Then we deduce by \eqref{eq:B.601} that
\begin{equation}\label{eq:B.12}
\begin{aligned}
&-\int_{\mathbb{T}^2}\Lambda^{\alpha/2}_{y}\mathbb{G}(t,x_0,y_0,x,y)
\Lambda^{-\alpha/2}_{y_0}\partial_{y_0}f_0(x_0,y_0)dx_0dy_0\\
=&-\Lambda^{\alpha/2}_{y}\int_{\mathbb{T}^2}\mathbb{G}(t,x_0,y_0,x,y)
\Lambda^{-\alpha/2}_{y_0}\partial_{y_0}f_0(x_0,y_0)dx_0dy_0\\
=&-\Lambda^{\alpha/2}_{y}\mathcal{S}_{t}\Lambda^{-\alpha/2}_{y}\partial_{y}f_0(t,x,y).
\end{aligned}
\end{equation}
Combining \eqref{eq:B.601}, \eqref{eq:B.8} and \eqref{eq:B.1101}-\eqref{eq:B.1301}, one has
$$
\begin{aligned}
J=&-\int_{\mathbb{T}^2}\left(\int_{0}^{t}e^{-(t-\tau)\mathcal{L}_{\nu,\alpha}}\mathcal{R}(\tau,x_0,y_0,x,y)d\tau\right)
\Lambda^{-\alpha/2}_{y_0}\partial_{y_0}f_0(x_0,y_0)dx_0dy_0\\
=&-C_{\alpha}\sum_{k\in\mathbb{Z}}
P.V.\int_{\mathbb{T}^2}\left(\int_{0}^{t}e^{-(t-\tau)\mathcal{L}_{\nu,\alpha}}\left(\int_{\mathbb{T}}\mathcal{K}(y,\widetilde{y})\partial_x\mathbb{G}
(\tau,x_0,y_0,x,\widetilde{y})d\widetilde{y}\right)d\tau\right)\\
&\ \ \ \ \ \ \ \cdot\Lambda^{-\alpha/2}_{y_0}\partial_{y_0}f_0(x_0,y_0)dx_0dy_0\\
=&C_{\alpha}\sum_{k\in\mathbb{Z}}
P.V.\int_{\mathbb{T}^2}\Lambda_{x_0}^{\alpha/2-1}\partial_{x_0}\left(\int_{0}^{t}e^{-(t-\tau)\mathcal{L}_{\nu,\alpha}}\left(\int_{\mathbb{T}}\mathcal{K}(y,\widetilde{y})\mathbb{G}
(\tau,x_0,y_0,x,\widetilde{y})d\widetilde{y}\right)d\tau\right)\\
&\ \ \ \ \ \ \ \cdot\Lambda^{1-\alpha/2}_{x_0}\Lambda^{-\alpha/2}_{y_0}\partial_{y_0}f_0(x_0,y_0)dx_0dy_0\\
=&-C_{\alpha}\sum_{k\in\mathbb{Z}}
P.V.\int_{\mathbb{T}^2}\left(\int_{0}^{t}e^{-(t-\tau)\mathcal{L}_{\nu,\alpha}}\Lambda_{x}^{\alpha/2-1}\partial_{x}
\left(\int_{\mathbb{T}}\mathcal{K}(y,\widetilde{y})\mathbb{G}
(\tau,x_0,y_0,x,\widetilde{y})d\widetilde{y}\right)d\tau\right)\\
&\ \ \ \ \ \ \ \cdot\Lambda^{1-\alpha/2}_{x_0}\Lambda^{-\alpha/2}_{y_0}\partial_{y_0}f_0(x_0,y_0)dx_0dy_0.
\end{aligned}
$$
By exchanging the order of integrals, one has
\begin{equation}\label{eq:B.16}
\begin{aligned}
J=&-C_{\alpha}\sum_{k\in\mathbb{Z}}
P.V.\int_{0}^{t}e^{-(t-\tau)\mathcal{L}_{\nu,\alpha}}\Lambda_{x}^{\alpha/2-1}\partial_{x}\int_{\mathbb{T}}\mathcal{K}(y,\widetilde{y})\\
&\cdot\left(\int_{\mathbb{T}^2}\mathbb{G}
(\tau,x_0,y_0,x,\widetilde{y})\Lambda^{1-\alpha/2}_{x_0}\Lambda^{-\alpha/2}_{y_0}\partial_{y_0}f_0(x_0,y_0)dx_0dy_0 \right)d\widetilde{y}d\tau\\
=&-C_{\alpha}\sum_{k\in\mathbb{Z}}P.V.\int_{0}^{t}e^{-(t-\tau)\mathcal{L}_{\nu,\alpha}}\int_{\mathbb{T}}\mathcal{K}(y,\widetilde{y})
\Lambda_{x}^{\alpha/2-1}\partial_{x}\mathcal{S}_{\tau}\left(\Lambda^{1-\alpha/2}_{x}\Lambda^{-\alpha/2}_{\widetilde{y}}
\partial_{\widetilde{y}}f_0\right)d\widetilde{y}d\tau\\
=&-\int_{0}^{t}e^{-(t-\tau)\mathcal{L}_{\nu,\alpha}}\mathcal{R}_{f_0}(\tau,x,y)d\tau,
\end{aligned}
\end{equation}
Combining \eqref{eq:B.11}-\eqref{eq:B.16}, we have
$$
\mathcal{S}_t\partial_yf_0=-\Lambda^{\alpha/2}_{y}\mathcal{S}_{t}\Lambda^{-\alpha/2}_{y}\partial_{y}f_0
-\int_{0}^{t}e^{-(t-\tau)\mathcal{L}_{\nu,\alpha}}\mathcal{R}_{f_0}(\tau,x,y)d\tau.
$$
This completes the proof of Lemma \ref{lem:B.4}.
\end{proof}

\medskip

\noindent{\bf Acknowledgments.}  The works of B.Niu and W.Wang were partially supported by the National Natural Science Foundation of China (Grant No.12271357,12161141004) and Shanghai Science and Technology Innovation Action Plan (Grant No.21JC1403600). The work of B.Shi was partially supported by the Jiangsu Funding Program for
Excellent Postdoctoral Talent (Grant No.2023ZB116). B.Shi also would like to thank Professor Xi Zhang for continued support and encouragement.

\medskip

\bibliographystyle{abbrv}
\bibliography{FKS-Shear-NSW-ref}
\medskip
\medskip
\end{document}